\documentclass[11pt]{article}
\usepackage{amsfonts,mathrsfs,amssymb,amsthm}
\usepackage{mathtools}

\usepackage{amsmath,amscd}
\usepackage{pslatex}
\usepackage{color}
\usepackage{hyperref}
\usepackage[all]{xy}
\usepackage{pdfpages}
\usepackage{tikz-cd}
\usepackage{graphicx}
\usepackage{fancyhdr}
\usetikzlibrary{calc}      
\usetikzlibrary{arrows.meta} 

\begin{document}
\newtheorem{Def}{Definition}[section]
\newtheorem{Bsp}[Def]{Example}
\newtheorem{Prop}[Def]{Proposition}
\newtheorem{Theo}[Def]{Theorem}
\newtheorem{Lem}[Def]{Lemma}
\newtheorem{Koro}[Def]{Corollary}
\theoremstyle{definition}
\newtheorem{Rem}[Def]{Remark}

\newcommand{\add}{{\rm add}}
\newcommand{\con}{{\rm con}}
\newcommand{\gd}{{\rm gldim}}
\newcommand{\qgd}{{\rm qgldim}}
\newcommand{\sd}{{\rm stdim}}
\newcommand{\sr}{{\rm sr}}
\newcommand{\dm}{{\rm domdim}}
\newcommand{\cdm}{{\rm codomdim}}
\newcommand{\tdim}{{\rm dim}}
\newcommand{\E}{{\rm E}}
\newcommand{\Mor}{{\rm Morph}}
\newcommand{\End}{{\rm End}}
\newcommand{\ind}{{\rm ind}}
\newcommand{\rsd}{{\rm resdim}}
\newcommand{\rd} {{\rm rd}}
\newcommand{\ol}{\overline}
\newcommand{\overpr}{$\hfill\square$}
\newcommand{\rad}{{\rm rad}}
\newcommand{\soc}{{\rm soc}}
\renewcommand{\top}{{\rm top}}
\newcommand{\pd}{{\rm pd}}
\newcommand{\id}{{\rm idim}}
\newcommand{\fld}{{\rm fdim}}
\newcommand{\Fac}{{\rm Fac}}
\newcommand{\Gen}{{\rm Gen}}
\newcommand{\fd} {{\rm findim}}
\newcommand{\qpd} {{\rm qpd}}
\newcommand{\Fd} {{\rm Findim}}
\newcommand{\Pf}[1]{{\mathscr P}^{<\infty}(#1)}
\newcommand{\DTr}{{\rm DTr}}
\newcommand{\cpx}[1]{#1^{\bullet}}
\newcommand{\D}[1]{{\mathscr D}(#1)}
\newcommand{\Dz}[1]{{\mathscr D}^+(#1)}
\newcommand{\Df}[1]{{\mathscr D}^-(#1)}
\newcommand{\Db}[1]{{\mathscr D}^b(#1)}
\newcommand{\C}[1]{{\mathscr C}(#1)}
\newcommand{\Cz}[1]{{\mathscr C}^+(#1)}
\newcommand{\Cf}[1]{{\mathscr C}^-(#1)}
\newcommand{\Cb}[1]{{\mathscr C}^b(#1)}
\newcommand{\Dc}[1]{{\mathscr D}^c(#1)}
\newcommand{\K}[1]{{\mathscr K}(#1)}
\newcommand{\Kz}[1]{{\mathscr K}^+(#1)}
\newcommand{\Kf}[1]{{\mathscr  K}^-(#1)}
\newcommand{\Kb}[1]{{\mathscr K}^b(#1)}
\newcommand{\DF}[1]{{\mathscr D}_F(#1)}
\newcommand{\Kac}[1]{{\mathscr K}_{\rm ac}(#1)}
\newcommand{\Keac}[1]{{\mathscr K}_{\mbox{\rm e-ac}}(#1)}
\newcommand{\modcat}{\ensuremath{\mbox{{\rm -mod}}}}
\newcommand{\Modcat}{\ensuremath{\mbox{{\rm -Mod}}}}
\newcommand{\Spec}{{\rm Spec}}
\newcommand{\stmc}[1]{#1\mbox{{\rm -{\underline{mod}}}}}
\newcommand{\Stmc}[1]{#1\mbox{{\rm -{\underline{Mod}}}}}
\newcommand{\prj}[1]{#1\mbox{{\rm -proj}}}
\newcommand{\inj}[1]{#1\mbox{{\rm -inj}}}
\newcommand{\Prj}[1]{#1\mbox{{\rm -Proj}}}
\newcommand{\Inj}[1]{#1\mbox{{\rm -Inj}}}
\newcommand{\PI}[1]{#1\mbox{{\rm -Prinj}}}
\newcommand{\GP}[1]{#1\mbox{{\rm -GProj}}}
\newcommand{\GI}[1]{#1\mbox{{\rm -GInj}}}
\newcommand{\gp}[1]{#1\mbox{{\rm -Gproj}}}
\newcommand{\gi}[1]{#1\mbox{{\rm -Ginj}}}
\newcommand{\opp}{^{\rm op}}
\newcommand{\otimesL}{\otimes^{\rm\mathbb L}}
\newcommand{\rHom}{{\rm\mathbb R}{\rm Hom}\,}
\newcommand{\pdim}{\pd}
\newcommand{\Hom}{{\rm Hom}}
\newcommand{\Coker}{{\rm Coker}}
\newcommand{ \Ker  }{{\rm Ker}}
\newcommand{ \Cone }{{\rm Con}}
\newcommand{ \Img  }{{\rm Im}}
\newcommand{\Ext}{{\rm Ext}}
\newcommand{\StHom}{{\rm \underline{Hom}}}
\newcommand{\StEnd}{{\rm \underline{End}}}
\newcommand{\KK}{I\!\!K}
\newcommand{\gm}{{\rm _{\Gamma_M}}}
\newcommand{\gmr}{{\rm _{\Gamma_M^R}}}
\def\vez{\varepsilon}\def\bz{\bigoplus}  \def\sz {\oplus}
\def\epa{\xrightarrow} \def\inja{\hookrightarrow}
\newcommand{\lra}{\longrightarrow}
\newcommand{\llra}{\longleftarrow}
\newcommand{\lraf}[1]{\stackrel{#1}{\lra}}
\newcommand{\llaf}[1]{\stackrel{#1}{\llra}}
\newcommand{\ra}{\rightarrow}
\newcommand{\dk}{{\rm dim_{_{k}}}}
\newcommand{\holim}{{\rm Holim}}
\newcommand{\hocolim}{{\rm Hocolim}}
\newcommand{\colim}{{\rm colim\, }}
\newcommand{\limt}{{\rm lim\, }}
\newcommand{\Add}{{\rm Add }}
\newcommand{\Prod}{{\rm Prod }}
\newcommand{\Tor}{{\rm Tor}}
\newcommand{\Cogen}{{\rm Cogen}}
\newcommand{\Tria}{{\rm Tria}}
\newcommand{\Loc}{{\rm Loc}}
\newcommand{\Coloc}{{\rm Coloc}}
\newcommand{\tria}{{\rm tria}}
\newcommand{\Con}{{\rm Con}}
\newcommand{\Thick}{{\rm Thick}}
\newcommand{\thick}{{\rm thick}}
\newcommand{\Sum}{{\rm Sum}}
\newcommand{\dep}{{\rm depth}}

\begin{center}
{\Large {\bf Properties of quasi-projective dimension over abelian categories}}
\end{center}

\centerline{\textbf{Hongxing Chen}$^*$, \textbf{Xiaohu Chen} and \textbf{Mengge Liu}}

\renewcommand{\thefootnote}{\alph{footnote}}
\setcounter{footnote}{-1} \footnote{ $^*$ Corresponding author.
Email: chenhx@cnu.edu.cn; Fax: 0086 10 68903637.}
\renewcommand{\thefootnote}{\alph{footnote}}
\setcounter{footnote}{-1} 
\renewcommand{\thefootnote}{\alph{footnote}}
\setcounter{footnote}{-1} 
\begin{abstract}

Quasi-projective dimension was introduced by Gheibi, Jorgensen and Takahashi to generalize the Auslander-Buchsbaum formula and the depth formula in commutative algebra. In this paper, we establish some basic properties of quasi-projective dimensions of objects in abelian categories. Analogous to global dimension of rings, we also introduce the concept of quasi-global dimension for left Noetherian rings, and then compare quasi-global dimension with global dimension for a class of Nakayama algebras. This provides new examples of finite-dimensional algebras with finite quasi-global dimensions but infinite global dimensions.

\end{abstract}


\section{Introduction}\label{Introduction}
In the representation theory of algebras and homological algebra, projective dimensions of objects in abelian categories  plays a very important role. In \cite{QPD}, Gheibi, Jorgensen and Takahashi introduced a generalization of projective dimension, called \emph{quasi-projective dimension}. This new dimension not only supplies a general framework to establish the Auslander–Buchsbaum formula and the depth formula in commutative algebra for modules of finite quasi-projective dimension (see \cite{2024}), but also gives an equivalent characterization of equipresented, local, complete intersection rings in term of finite quasi-projective dimension (see \cite{2021}). Dual to quasi-projective dimensions, quasi-injective dimensions of modules over rings were introduced to provide extensions of the Bass's formula and the Chouinard's formula in commutative algebra for modules of finite quasi-injective dimensions (see \cite{G2024,NMT}).

We start from recalling some elementary properties of quasi-projective dimensions of objects in abelian categories established in \cite{QPD}. \emph{Throughout this section, let $\mathcal{A}$ be an arbitrary abelian category with enough projective objects}. For each object $M\in\mathcal{A}$, we denote by ${\pd}_{\mathcal{A}}\,M$, $\qpd_{\mathcal{A}}\,M$ and $\Omega^n(M)$ its projective dimension, quasi-projective dimension and $n$-th syzygy for $n\geqslant 0$ in $\mathcal{A}$, respectively. For the definitions of quasi-projective resolution and quasi-projective dimension, we refer to Definition \ref{def3.1(1)}.

\begin{Theo}{\rm \cite[Propositions 3.3 and 3.6(2)]{QPD}}\label{QPP}
The following statements hold for an object $M\in\mathcal{A}$.

$(1)$ $\qpd_{\mathcal{A}}(M\oplus N)\leqslant\sup\{\qpd_{\mathcal{A}}\,M,\,\qpd_{\mathcal{A}}\,N\}$ for any object $N\in\mathcal{A}$. In particular, if $N$ is projective, then $\qpd_{\mathcal{A}}\,(M\oplus N)\leqslant\qpd_{\mathcal{A}}\,M$.

$(2)$ $\qpd_{\mathcal{A}}\,\Omega(M)\leqslant\qpd_{\mathcal{A}}\,M$.

$(3)$ If $M$ is periodic, that is, there exists a positive integer $r$ such that $\Omega^r(M)\simeq M$, then $\qpd_{\mathcal{A}}\,M=0$.
\end{Theo}

From Theorem \ref{QPP}, we can see some obvious differences between projective dimension and quasi-projective dimension. For example, the inequalities in Theorem \ref{QPP}(1) are equalities if quasi-projective dimension is replaced by projective dimension, while a periodic, non-projective object has infinite projective dimension.

The paper is devoted to providing more homological properties for quasi-projective dimensions of objects in abelian categories. Our main result reads as follows.

\begin{Theo}\label{main1.1}
Let $\mathcal{A}$ be an abelian category with enough projective objects. The following statements hold for an object $M\in\mathcal{A}$.

$(1)$ If ${\pd}_{\mathcal{A}}\,M<\infty$, then $\qpd_{\mathcal{A}}\,M={\pd}_{\mathcal{A}}\,M$.

$(2)$ If $\qpd_{\mathcal{A}}\,M<\infty$ and $\Ext_\mathcal{A}^n(M,M)=0$ for all $n\geqslant 2$, then ${\rm pd}_{\mathcal{A}}\,M<\infty$.

$(3)$ If $0\to N\to E\to M\to 0$ is an exact sequence in ${\mathcal{A}}$ such that $E$ is projective-injective, then $$\qpd_{\mathcal{A}}\,M\leqslant\qpd_{\mathcal{A}}\, N+1.$$

$(4)$ Let $P_{\bullet}$ be a deleted projective resolution of $M$ in $\mathcal{A}$. Suppose that there is an integer $n\geqslant 2$ and a chain map $g_{\bullet}:P_{\bullet}\to P_{\bullet}[n]$ of complexes over $\mathcal{A}$:
$$\begin{tikzcd}
\cdots \arrow[r] & P_{n+1} \arrow[r] \arrow[d, "g_{n+1}"] & P_n \arrow[r] \arrow[d, "g_n"] & P_{n-1} \arrow[r] \arrow[d, "0"]  & \cdots \arrow[r] & P_0 \arrow[r] \arrow[d, "0"] & 0 \\
\cdots \arrow[r] & P_1 \arrow[r]                          & P_0 \arrow[r]                  & 0 \arrow[r]                 & \cdots \arrow[r] & 0 \arrow[r]             & 0
\end{tikzcd}$$
inducing an isomorphism $\Omega^{n+m}(M)\to\Omega^m(M)$ for some integer $m\geqslant 0$. Then $\qpd_{\mathcal{A}}\,M\leqslant m$.
\end{Theo}

Note that Theorem \ref{main1.1}(1) and Theorem \ref{main1.1}(2) have been shown for quasi-projective dimensions of finitely generated modules over commutative Noetherian rings by using some techniques from commutative algebra; see \cite[Corollary 4.10 and Theorem 6.20]{QPD}. Theorem \ref{main1.1}(3)(4) are completely new (even in the case of commutative rings) and provide two effective methods for bounding quasi-projective dimensions (see Example \ref{newcase}). Note that Theorem \ref{main1.1}(3) does not hold when the object $E$ is projective but not injective; see Example \ref{Counterexample} for a counterexample.

An interesting application of Theorem \ref{main1.1}(2) is the validity of \emph{the Auslander–Reiten conjecture}  for modules of finite quasi-projective dimension over left Noetherian rings. Recall that the conjecture says that if a finitely generated $R$-module $M$ over a left Noetherian ring $R$ satisfies $\Ext_R^{n}(M,M\oplus R)=0$ for all $n\geqslant 1$, then it is projective.

We formulate the following corollary for left coherent rings which are a generalization of left Noetherian rings.
\begin{Koro}\label{main1.2}
Let $R$ be a left coherent ring. Suppose that $M$ is a finitely presented left $R$-module with finite quasi-projective dimension. If ${\rm Ext}_{R}^n(M,M\oplus R)=0$ for all $n\geqslant 1$, then $M$ is projective.
\end{Koro}

Analogous to global dimension, we introduce the notion of quasi-global dimension for left Noetherian rings and establish some basic properties of this new dimension in Proposition \ref{basic}. In particular, we show that \emph{Tachikawa's second conjecture holds for any self-injective Artin algebra with finite quasi-global dimension (or equivalently, with quasi-global dimension $0$)}; see Corollary \ref{QDZO} and Remark \ref{TSC}. It is worth mentioning that if two finite-dimensional self-injective algebras over a field are stably equivalent of Morita type or derived equivalent, then they have the same quasi-global dimension (see Corollary \ref{moritatype}). For other new advances on Tachikawa's second conjecture, we also refer the reader to \cite{xcf,xc}.

Further, using Theorem \ref{main1.1}, we can also determine the quasi-global dimensions of a class of finite-dimensional Nakayama algebras (see Theorem \ref{main} for more details). Consequently, those algebras always have finite quasi-global dimension, although they may have infinite global dimension in some cases.

\medskip
The paper is organized as follows. In Section \ref{Preli}, we recall some definitions and basic facts on quasi-projective dimensions of objects in abelian categories. In Section \ref{3}, we show Theorem \ref{main1.1}. In Section \ref{4}, we discuss quasi-global dimension for left Noetherian rings and calculate this dimension for a class of Nakayama algebras.

\section{Preliminaries}\label{Preli}
In this section, we introduce some standard notation and recall the definitions of quasi-projective resolution and quasi-projective dimensions of objects in abelian categories from \cite{QPD}.

All rings considered in this paper are assumed to be associative and with identity and all modules are unitary left modules.

Let $R$ be a ring. We denote by $R\Modcat$ the category of all $R$-modules and by $R\modcat$ the category of all finitely presented $R$-modules. For an $R$-module $M$, let {\rm add}$(M)$ (respectively, {\rm Add}$(M)$) be the full subcategory of $R$-{\rm Mod} consisting of all direct summands of ﬁnite (respectively, arbitrary) direct sums of copies of $M$. In many circumstances, we write $R$-{\rm proj} and $R$-{\rm Proj} for {\rm add}$({}_RR)$ and {\rm Add}$({}_RR)$, respectively. When $R$ is \emph{left coherent}, that is, every finitely generated left ideal of $R$ is finitely presented, the category $R\modcat$ is an abelian category and has $\prj{R}$ as its full subcategory consisting of all projective objects.

Let $\mathcal{A}$ be an additive category. By a \emph{complex} $X_{\bullet}\coloneqq(X_i,d_i^X)_{i\in\mathbb{Z}}$ over $\mathcal{A}$, we mean a sequence of morphisms $d^X_i$ between objects $X_i$ in $\mathcal{A}$: $$\begin{tikzcd}
\cdots \arrow[r, "d_{i+2}^X"] & X_{i+1} \arrow[r, "d_{i+1}^X"] & X_i \arrow[r, "d_{i}^X"] & X_{i-1} \arrow[r, "d_{i-1}^X"] & \cdots,
\end{tikzcd}$$such that $d_i^Xd_{i+1}^X=0$ for $i\in\mathbb{Z}$. For simplicity, we sometimes write $(X_i)_{i\in\mathbb{Z}}$ for $X_{\bullet}$ without mentioning $d_i^X$. For a fixed integer $n$, we denote by $X_{\bullet}[n]$ the complex obtained from $X_{\bullet}$ by shifting $n$ degrees, that is, $(X_{\bullet}[n])_i=X_{i-n}$. Let $\mathscr{C}(\mathcal{A})$ be the category of all complexes over $\mathcal{A}$ with chain maps, and let $\mathscr{K}(\mathcal{A})$ be the homotopy category of $\mathscr{C}(\mathcal{A})$. For a chain map $f_{\bullet}:X_{\bullet}\to Y_{\bullet}$ in $\mathscr{C}(\mathcal{A})$, we denoted by $ \Cone(f_{\bullet})$ the mapping cone of $f_{\bullet}$.
When $\mathcal{A}$ is an abelian category, we also define the $i$-th cycle, boundary and homology of $X_{\bullet}$ for each $i$ as follows:
$$
Z_i(X_{\bullet})\coloneqq\Ker(d_i^X),\quad  B_i(X_{\bullet})\coloneqq\Img(d_{i+1}^X)\quad\mbox{and}\quad  H_{i}(X_{\bullet})\coloneqq Z_i(X_{\bullet})/B_i(X_{\bullet}).
$$
Moreover, the \emph{supremum}, \emph{infimum}, \emph{homological supremum} and \emph{homological infimum} of $X_{\bullet}$ are defined by
$$\sup(X_{\bullet})\coloneqq\sup\{i\in\mathbb{Z}\,|\,X_i\not=0\},\quad \inf(X_{\bullet})\coloneqq\inf\{i\in\mathbb{Z}\,|\,X_i\not=0\},$$
$$\text{hsup}(X_{\bullet})\coloneqq\sup\{i\in\mathbb{Z}\,|\,H_i(X_{\bullet})\not=0\},
\quad\text{hinf}(X_{\bullet})\coloneqq\inf\{i\in\mathbb{Z}\,|\,H_i(X_{\bullet})\not=0\}.$$
Clearly, $\inf(X_{\bullet})\leqslant\text{hinf}(X_{\bullet})\leqslant\text{hsup}(X_{\bullet})\leqslant\sup(X_{\bullet}) $. We say that $X_{\bullet}$ is \emph{bounded} (respectively, \emph{bounded below}) if $\sup(X_{\bullet})<\infty$ and $\inf(X_{\bullet})>-\infty$ (respectively, ${\rm inf}(X_{\bullet})>-\infty$).
Let $\mathscr{C}^{-}(\mathcal{A})$ and $\mathscr{C}^{{\rm b}}(\mathcal{A})$ be the categories of all bounded below and bounded complexes over $\mathcal{A}$, respectively. Their homotopy categories are denoted by $\mathscr{K}^{-}(\mathcal{A})$ and $\mathscr{K}^{\mathrm{b}}(\mathcal{A})$, respectively.

When $\mathcal{A}$ is an abelian category, the bounded below derived category of $\mathcal{A}$ is denoted by $\mathscr{D}^{-}\mathcal{(A)}$, which is the localization of $\mathscr{K}^{-}(\mathcal{A})$ at all quasi-isomorphisms. Moreover, the \emph{extension} of two full subcategories $\mathcal{X}$ and $\mathcal{Y}$ in $\mathscr{D}^{-}\mathcal{(A)}$, denoted by $\mathcal{X} \ast \mathcal{Y}$, is defined to be the full subcategory of $\mathscr{D}^{-}\mathcal{(A)}$ consisting of all objects $Z$ such that there exists a distinguished triangle $X \to Z \to Y \to X[1]$ in $\mathscr{D}^{-}\mathcal{(A)}$ with $X\in \mathcal{X}$ and $Y\in \mathcal{Y}$. Normally, we identify $\mathcal{A}$ with the full subcategory of $\mathscr{D}^{-}(\mathcal{A})$ consisting of all stalk complexes concentrated in degree $0$.

From now on, \emph{let $\mathcal{A}$ be an abelian category with enough projective objects}.

Let $M$ be an object of $\mathcal{A}$. We denote by ${\rm pd}_{\mathcal{A}}\,M$ the projective dimension of $M$ in $\mathcal{A}$. Let
$$P_{\bullet}:\;\;\cdots\lra P_{i+1}\lra P_i\lra P_{i-1}\lra\cdots\lra P_1\lra P_0\lra 0$$ be a deleted projective resolution of $M$ in $\mathcal{A}$. For each integer $i\geqslant 1$, the $i$-th {\rm syzygy} of $M$ (with respect to $P_{\bullet}$) is defined to be $\Omega^i(M):=B_{i-1}(P_{\bullet})$. As usual, we write $\Omega(M)$ for $\Omega^1(M)$, and understand that $\Omega^0(M)=M$. The object $M$ is said to be \emph{periodic} if there exists an integer $r\geqslant 1$ such that $\Omega^r(M)\simeq M$. When $\mathcal{A}=R\Modcat$, we simply write $\pd_R\,M$ for $\pd_{\mathcal{A}}\,M$.

The following definition is taken from \cite[Definition 3.1]{QPD}.

\begin{Def}\label{def3.1(1)}
$(1)$ A complex $P_{\bullet}\in\mathscr{C}^{-}(\mathcal{A})$ is called a \emph{quasi-projective resolution} of $M$ if the following two conditions hold:

$(a)$ all $P_i$ are projective for $i\in\mathbb{Z}$;

$(b)$ for all integers $j\geqslant{\rm inf}(P_{\bullet})$, there are integers $n_j\geqslant 0$, not all zero, such that $H_j(P_{\bullet})\simeq M^{n_j}$. If, in addition, ${\rm sup}(P_{\bullet})<\infty$ (that is, $P_{\bullet}\in\mathscr{C}^{{\rm b}}(\mathcal{A}))$, then $P_{\bullet}$ is called a \emph{finite quasi-projective resolution} of $M$.

$(2)$ The \emph{quasi-projective dimension} of a nonzero object $M$ in $\mathcal{A}$ is defined to be
$$
\qpd_{\mathcal{A}}\,M\coloneqq\inf\{{\rm sup}(P_{\bullet})-{\rm hsup}(P_{\bullet})|\,P_{\bullet}\,\text{\,is a finite quasi-projective resolution of}\,\,M\}.
$$
We understand that the quasi-projective dimension of the zero object is $0$ (compared with  {\rm \cite[Definition 3.1(2)]{QPD}}).
\end{Def}

By Definition \ref{def3.1(1)}, $\qpd_{\mathcal{A}}\,M=\infty$ if and only if $M$ does not have a finite quasi-projective resolution. Since each deleted projective resolution of $M$ is automatically a quasi-projective resolution, it is clear that $\qpd_{\mathcal{A}}\,M\leqslant{\rm pd}_{\mathcal{A}}M$. The inequality can be strict, for example, if $M$ is periodic but non-projective (see Theorem \ref{QPP}(3)). Moreover, by shifting quasi-projective resolutions, we see that $\qpd_{\mathcal{A}}(M)$ is the infimum
of the projective dimensions ${\pd}_{\mathcal{A}}(\Coker(d_1^P)$, where  $P_{\bullet}\coloneqq(P_i,d_i^P)_{i\in\mathbb{Z}}$ runs through all finite quasi-projective resolutions of $M$ with $\mathrm{hsup}\,P_{\bullet}=0$. In this paper, for a ring $R$ and any $R$-module $M$, we set $\qpd_R\,M\coloneqq\qpd_{R\text{-Mod}}\,M$.

The following result reveals a close relationship between objects with finite quasi-projective dimensions and objects with finite projective dimensions.

\begin{Prop}\label{general 3}
Let $\mathcal{P}$ be the full subcategory of $\mathcal{A}$ consisting of (not necessarily all) projective objects and being closed under direct summands and finite direct sums. Let $M$ be an object of
$\mathcal{A}$ such that $\qpd_{\mathcal{A}}(M)<\infty$. Suppose that $M$ admits a projective resolution of the form
\[
\begin{tikzcd}
Q_{\bullet}:\quad\cdots\arrow[r]& Q_2 \arrow[r]  & Q_1 \arrow[r] & Q_0 \arrow[r] & M \arrow[r] & 0
\end{tikzcd}
\]
with all $Q_i\in\mathcal{P}$ for $i\geqslant 0$. Then there exists a bounded complex
\[
\begin{tikzcd}
P_{\bullet}:\quad0 \arrow[r] & P_r \arrow[r, "d_r"] & \cdots \arrow[r, "d_2"] & P_1 \arrow[r, "d_1"] & P_0 \arrow[r, "d_0"] & P_{-1} \arrow[r] &\cdots \arrow[r, "d_{-s+1}"] & P_{-s} \arrow[r] & 0
\end{tikzcd}
\]
in $\mathscr{C}^{\rm b}(\mathcal{P})$ with $s\geqslant 0$ and $\mathrm{hsup}\,P_{\bullet}=0$ satisfying the following:

$(a)$ The complex $P_{\bullet}$ is a quasi-projective resolution of $M$ and $\qpd_{\mathcal{A}}\,M=\pd_{\mathcal{A}}\,N$, where $N\coloneqq\Coker(d_1)$.

$(b)$ For a right exact, covariant functor $F:\mathcal{A}\to\mathcal{B}$ between abelian categories, let $L^iF$ denote the $i$-th left-derived functor of $F$ and let $d\coloneqq{\rm sup}\{i\geqslant0\,|\,L^iF(M)\not=0\}$.

If $0\leqslant d<\infty$, then $d=\sup\{i\geqslant 0\,|\,L^iF(N)\not=0\}$. Furthermore, if $0<d<\infty$, then $L^dF\big(H_0(P_{\bullet})\big)\simeq L^dF(N)$.
\end{Prop}
\begin{proof}
Let $\mathcal{X}$ denote the full subcategory of $\mathcal{A}$ consisting of all projective objects. Then $\mathcal{P}\subseteq \mathcal{X}$. Let $r\coloneqq\qpd_{\mathcal{A}} M < \infty$. Then $M$ has a finite quasi-projective resolution $S_{\bullet}\coloneqq(S_i,d_i^S)_{i\in\mathbb{Z}}\in\mathscr{C}^{\rm b}(\mathcal{X})$ in $\mathcal{A}$ with $\mathrm{hsup}(S_{\bullet})=0$ and $r={\pd}_{\mathcal{A}}\big(\Coker(d_1^S)\big)$. Set $s\coloneqq-\mathrm{hinf}(S_{\bullet})$. Then $0\leqslant s<\infty$. By taking truncations of complexes over $\mathcal{A}$, we have
\[
S_{\bullet}\in H_0(S_{\bullet}) \ast H_{-1}(S_{\bullet})[-1] \ast \cdots \ast H_{-s}(S_{\bullet})[-s]\subseteq \mathscr{D}^{-}(\mathcal{A}).
\]
Clearly, for each $i$, $H_i(S_{\bullet})\simeq M^{n_i}$ for some integer $n_i \geqslant 0$.
Since $\mathcal{P}\subseteq \mathcal{A}$ is closed under finite direct sums, we can replace $M$ with its projective resolution $Q_{\bullet}\in\mathscr{C}^{-}(\mathcal{P})$ and construct a complex $T_{\bullet} \in \mathscr{C}^{-}(\mathcal{P})$ that is isomorphic in $\mathscr{D}^{-}(\mathcal{A})$ to $S_{\bullet}$. Consequently, $H^i(T_{\bullet})\simeq H^i(S_{\bullet})$ for each $i\in\mathbb{Z}$.
This implies that $T_{\bullet}$ is also a quasi-projective resolution of $M$ with $\mathrm{hsup}(T_{\bullet})=0$. Further, since $\mathscr{K}^{-}(\mathcal{X})$ is triangle equivalent to $\mathscr{D}^{-}(\mathcal{A})$ and $\mathscr{K}^{-}(\mathcal{P})\subseteq \mathscr{K}^{-}(\mathcal{X})$, we have $T_{\bullet} \simeq S_{\bullet}$ in $\mathscr{K}^{-}(\mathcal{X})$. It follows that there exist contractible complexes $X_{\bullet}$ and $Y_{\bullet}$ in $\mathscr{C}^{-}(\mathcal{X})$ such that $T_{\bullet} \oplus X_{\bullet} \simeq S_{\bullet} \oplus Y_{\bullet}$ in $\mathscr{C}^{-}(\mathcal{X})$.
In both sides of the isomorphism, we take the cokernels of complexes in degree $1$ and obtain
$$\Coker(d_1^T)\oplus\Coker(d_1^X)\simeq\Coker(d_1^S)\oplus\Coker(d_1^Y).$$
Since $X_{\bullet}$ and $Y_{\bullet}$ are contractible, both $\Coker(d_1^X)$ and $\Coker(d_1^Y)$ belong to $\mathcal{X}$.
Thus $\pd(\Coker(d_1^T))=\pd(\Coker(d_1^S))$.
By $\mathrm{hsup}(T_{\bullet})=0$, the sequence
$$\cdots \lra T_r\lra T_{r-1}\lra \cdots\lra T_1\lra  T_0\lra\Coker(d_1^T)\lra 0,$$ truncated from the complex $T_{\bullet}$, is a projective resolution of $\Coker(d_1^T)$. It follows that $\Omega^r(\Coker(d_1^T))$ is projective, and thus isomorphic to a direct summand of $T_r$. Since $\mathcal{P}$ is closed under direct summands in $\mathcal{A}$ and contains $T_r$, we have $\Omega^r(\Coker(d_1^T))\in\mathcal{P}$.
Now, let $P_{\bullet}$ be the truncation of $T_{\bullet}$ in degree $r$:
$$
P_{\bullet}: \;0\lra \Omega^r(\Coker(d_1^T))\lra T_{r-1}\to T_{r-2}\lra \cdots\lra T_1\lra T_0\lra\cdots\lra T_{-s}\lra 0.
$$
Then $P_{\bullet}\in \mathscr{C}^{\rm b}(\mathcal{P})$ is homotopy equivalent to $T_\bullet$ and satisfies $(a)$.

To show $(b)$, suppose $0\leqslant d<\infty$. If $s=0$, then $N\simeq H_0(P_\bullet)\simeq M^{n_0}$ and thus $(b)$ holds trivially. In the following, let $s>0$. Observe that $H_0(P_{\bullet})\simeq H_0(S_{\bullet})\simeq M^{n_0}\neq 0$. Applying the functor $L^{d+j}F$ for $j\geqslant 0$ to the exact sequence $ 0 \to M^{n_0} \to N \to \Img(d_0)\to 0$ in $\mathcal{A}$ yields an exact sequence in $\mathcal{B}$:
\[ L^{d+j+1}F\big(\Img(d_0)\big)\lra L^{d+j}F(M^{n_0}) \lra L^{d+j}F(N) \lra L^{d+j}F(\Img\,d_0). \]
Consequently,  if $L^{d+j}F\big(\Img(d_0)\big) = 0$ for all $j > 0$, then
$$0=\big(L^{d+j}F(M)\big)^{n_0}\simeq L^{d+j}F(M^{n_0})\simeq L^{d+j}F(N)\quad \mbox{and}\quad 0\neq L^{d}F(M^{n_0})\hookrightarrow L^{d}F(N),$$ and therefore $d=\sup\{i \geqslant0 \mid L^iF(N) \neq 0\}$.

It remains to show $L^{d+j}F\big(\Img(d_0)\big) = 0$ for all $j > 0$.

Now, we fix $j>0$. Since $L^{d+k}F(M)=0$ for all $k>0$ and $H_{i-1}(P_{\bullet})\simeq M^{n_{i-1}}$ for $i\in\mathbb{Z}$, we apply the functor $L^{d+j}F$ to the exact sequence $0 \to \Img(d_i) \to \Ker(d_{i-1}) \to H_{i-1}(P_{\bullet}) \to 0 $ and obtain  $$L^{d+j}F(\Img(d_i))\simeq L^{d+j}F(\Ker(d_{i-1}))\;\;\mbox{and}\;\; L^dF(\Img(d_i))\hookrightarrow L^dF(\Ker(d_{i-1})).$$
Since $\Ker(d_{i-1})=\Omega(\Img(d_{i-1}))$ by $P_{i-1}\in\mathcal{P}$, it follows that
$L^{m}F(\Ker(d_{i-1}))\simeq L^{m+1}F(\Img(d_{i-1}))$ for all $m>0$. Thus
$L^{d+j}F(\Img(d_i))\simeq L^{d+j+1}F(\Img(d_{i-1}))$, and $L^dF(\Img(d_i))\hookrightarrow L^dF(\Ker(d_{i-1}))\simeq L^{d+1}F(\Img(d_{i-1}))$ for $d>0$. This implies a series of isomorphisms:
$$
L^{d+j}F(\Img(d_0))\simeq L^{d+j+1}F(\Img(d_{-1}))\simeq L^{d+j+2}F(\Img(d_{-2}))\simeq\cdots\simeq L^{d+j+s}F(\Img(d_{-s}))=0,$$
where the last equality follows from the vanish of the map $d_{-s}:P_{-s}\to 0$.

Suppose $d > 0$. Then
$$
L^dF(\Img(d_0))\hookrightarrow L^{d+1}F(\Img(d_{-1}))\simeq L^{d+2}F(\Img(d_{-2}))\simeq \cdots\simeq L^{d+s}F(\Img(d_{-s}))=0,
$$
and therefore $L^dF(\Img(d_0))=0$. From the exact sequence
\[ 0=L^{d+1}F(\mathrm{Im}(d_0)) \lra  L^dF(M^{n_0}) \lra L^dF(N) \lra L^dF(\mathrm{Im}(d_0))= 0, \]
we see that $L^dF(H_{0}(P_{\bullet}))\simeq L^dF(M^{n_0})\simeq L^dF(N).$ This shows $(b)$.

Similarly, we can also show that if $d=-\infty$ (that is, $L^iF(M)=0$ for all $i\geqslant 0$), then
\[
{\rm sup}\{i\in\mathbb{N}\,|\,L^iF(N)\not=0\} =
\begin{cases}
0, & \text{if } F(N)\neq 0; \\
-\infty, & \text{if } F(N)=0,
\end{cases}
\]
where the supreme is $-\infty$ means that $L^iF(N)=0$ for all $i\geqslant 0$.
\end{proof}

A direct consequence of Proposition \ref{general 3} is the following result which generalizes  \cite[Proposition 3.4]{QPD}.

\begin{Koro} \label{general 3.4}
Let $R$ be a left coherent ring and $M$ a finitely presented $R$-module with finite quasi-projective dimension. Then:

$(1)$  There exists a bounded complex $P_{\bullet}\in \mathscr{K}^{{\rm b}}(R\text{-}{\rm proj})$ that is a quasi-projective resolution of $M$ such that $\qpd_R(M)={\rm sup}(P_{\bullet})-{\rm hsup}(P_{\bullet})$.

$(2)$ There exists a finitely presented $R$-module $N$ such that $\pd_R(N)=\qpd_R(M)$ and there is an injection $M \hookrightarrow N$ of $R$-modules.
\end{Koro}
\begin{proof}
Since $R$ is left coherent, $R\modcat$ is an abelian category and has $\prj{R}$ as its full subcategory consisting of all projective objects. In Proposition \ref{general 3}, we take $\mathcal{A}=R\modcat$ and $\mathcal{P}=\prj{R}$. By Proposition \ref{general 3}$(a)$, there exists a bounded complex $P_{\bullet}\coloneqq(P_i,d_i^P)_{i\in\mathbb{Z}}\in\mathscr{C}^{{\rm b}}(R\text{-}{\rm proj})$ satisfying $(1)$ and ${\rm hsup}(P_{\bullet})=0$. Let $N\coloneqq\Coker(d_1^P)$. Then $N\in R\modcat$ and $\pd_R(N)=\qpd_R(M)$. Since there is an exact sequence $0\to H_0(P_\bullet)\to N\to \Img(d_0^P)\to 0$ of $R$-modules and $H_0(P_\bullet)\simeq M^{n_0}$ for some integer $n_0>0$, we obtain an injection $M \hookrightarrow N$ in $R\modcat$.
\end{proof}

\begin{Koro}\label{FB}
Let $\mathcal{A}$ be a Frobenius abelian category. Then $\qpd_\mathcal{A}(M)=0$ or $\infty$ for any $M\in\mathcal{A}$.
\end{Koro}

\begin{proof}
Since $\mathcal{A}$ is a Frobenius category, projective objects and injective objects of $\mathcal{A}$ coincide. This implies that objects of $\mathcal{A}$ with finite projective dimensions are exactly projective objects. By Proposition \ref{general 3}$(a)$, if $\qpd_\mathcal{A}(M)<\infty$, then $\qpd_\mathcal{A}(M)=0$.
\end{proof}

Finally, we give an application of Proposition \ref{general 3} to generalize the classical depth formula
(see \cite[Theorem 1.2]{1961auslander}). The following definition is standard in commutative algebra.

\begin{Def}
Let $(R, \mathfrak{m})$ be a commutative Noetherian local ring with the maximal ideal $\mathfrak{m}$. The \emph{depth} of a finitely generated $R$-module $M$ is defined as
\[\dep_{R}(M)\coloneqq\min\{i\geqslant 0\,\mid\,\Ext_{R}^{i}(R/\mathfrak{m},M)\not=0\}.\]
\end{Def}

\begin{Theo}{\rm \cite[Theorem 4.2]{2024}}\label{depth formula}
Let $R$ be a commutative Noetherian local ring, and let $M$ and $N$ be non-zero finitely generated $R$-modules. Define $q:=\sup\{i\geqslant 0\,\mid\,\Tor_{i}^{R}(M,N)\not=0\}$. Suppose that $q<\infty$ and $\qpd_{R}(M)<\infty$. If $q=0$ or $\dep(\Tor_{q}^{R}(M,N))\leqslant1$, then $$\dep(N)=\dep(\Tor_{q}^{R}(M,N))+\qpd_{R}(M)-q.$$
\end{Theo}
\begin{proof}
The case $q=0$ in Theorem \ref{depth formula} (that is, $\Tor_{i}^{R}(M,N)=0$ for all $i>0$) is a direct consequence of \cite[Theorems 4.4 and 4.11]{QPD}.

For the case $q>0$, we provide a new proof that is different from the one given in \cite[Theorem 4.2]{2024}.

Let $F:=-\otimes_{R}N: R\Modcat\to R\Modcat$, a covariant and right exact functor.
It is known that, for each $i\in\mathbb{N}$, the $i$-th left-derived functor $L^iF$ of $F$
is exactly the functor $\Tor^R_i(-, N)$. This implies
$$q=\sup\{i\geqslant 0\,\mid\,L^i(F)(M)\not=0\}<\infty.$$
Since $\qpd_{R}(M)<\infty$, it follows from Proposition \ref{general 3}$(b)$ that there exists a finite quasi-projective resolution $(P_{\bullet},d_{i})_{i\in\mathbb{Z}}$ of ${_R}M$ with $\mathrm{hsup}\,P_{\bullet}=0$ and $C:=\Coker(d_1)$ satisfying
$$q=\sup\{i\geqslant 0\,\mid\,\Tor_{i}^{R}(C,N)\not=0\},\;\; \pd_{R}(C)=\qpd_{R}(M)<\infty\;\;\mbox{and}\;\;
\Tor_{q}^{R}(C,N)\simeq\Tor_{q}^{R}(H_{0}(P_{\bullet}),N).$$
Note that $H_{0}(P_{\bullet})\simeq M^{n_0}$ for some $n_{0}>0$. This implies $\dep(\Tor_{q}^{R}(C,N))=\dep(\Tor_{q}^{R}(M,N))\leqslant 1$. Now, we apply \emph{Auslander's depth formula} (see \cite[Theorem 1.2]{1961auslander}) directly to the pair $(C,N)$ of $R$-modules, and obtain the formula $$\dep(N)=\dep(\Tor_{q}^{R}(C,N))+\pd_{R}(C)-q.$$ Thus $\dep(N)=\dep(\Tor_{q}^{R}(M,N))+\qpd_{R}(M)-q.$
\end{proof}

\section{New properties of quasi-projective dimensions}\label{3}

In this section, we develop some general methods to determine the quasi-projective dimensions of objects in abelian categories. In particular, we give a proof of Theorem \ref{main1.1}.

Throughout this section, $\mathcal{A}$ stands for an abelian category with enough projective objects. Denote by $\mathcal{P}$ the full subcategory of $\mathcal{A}$ consisting of \emph{all} projective objects.

Our first result generalizes {\rm \cite[Corollary 4.10]{QPD}} that is focused on finitely generated modules over commutative Noetherian rings and is a corollary of the \emph{Auslander–Buchsbaum formula} for modules of finite quasi-projective dimension established in \cite[Theorem 4.4]{QPD}.

\begin{Prop}\label{qp=p1}
Let $M$ be an object of $\mathcal{A}$ with ${\rm pd}_{\mathcal{A}}\,M<\infty$. Then $\qpd_{\mathcal{A}}\,M={\rm pd}_{\mathcal{A}}\,M$.
\end{Prop}
\begin{proof}
Clearly, $\qpd_{\mathcal{A}}\,M\leqslant{\rm pd}_{\mathcal{A}}\,M<\infty$. It suffices to show
${\rm pd}_{\mathcal{A}}\,M\leqslant\qpd_{\mathcal{A}}\,M$. Let $N$ be the object given in Proposition \ref{general 3}$(a)$. Then $\qpd_{\mathcal{A}}\,M=\pd_{\mathcal{A}}\,N$. We claim
${\rm pd}_{\mathcal{A}}\,M\leqslant \pd_{\mathcal{A}}\,N$.

In fact, for each object $X\in\mathcal{A}$, the contravariant functor $F_X\coloneqq\Hom_\mathcal{A}(-,X):\mathcal{A}\to \mathbb{Z}\Modcat$ can be regarded as a right exact, covariant functor from $\mathcal{A}$ to $(\mathbb{Z}\Modcat)\opp$ (that is the opposite category of $\mathbb{Z}\Modcat$). Then, for each $i\geqslant 0$, the $i$-th left-derived functor $L^iF_X$ of $F_X$ is exactly the $i$-th extension functor $\Ext_\mathcal{A}^i(-, X)$. Let $d_X\coloneqq\sup\{i\geqslant0\,|\,\Ext_\mathcal{A}^i(M, X)\not=0\}$.
By Proposition \ref{general 3}$(b)$, if $0\leqslant d_X<\infty$, then $d_X=\sup\{i\geqslant0\,|\,\Ext_\mathcal{A}^i(N, X)\not=0\}$. Let $n\coloneqq{\rm pd}_{\mathcal{A}}\,M<\infty$. Then $\Ext_\mathcal{A}^i(M, -)=0$ for $i>n$ and there exists an object $X\in\mathcal{A}$ such that $\Ext_\mathcal{A}^n(M, X)\neq 0$. This implies $d_X=n$. Thus $\Ext_\mathcal{A}^n(N, X)\not=0$ which forces
${\rm pd}_{\mathcal{A}}\,N\geqslant n$.
\end{proof}

\begin{Lem}\label{finitesum}
Let $M$ be an object of $\mathcal{A}$ with $\Ext_\mathcal{A}^n(M, M) = 0$ for
$2 \leqslant n \leqslant k+1$ with $k\in\mathbb{N}$. Suppose that $X_{\bullet}\in\mathscr{D}^{\rm b}(\mathcal{A})$ satisfies that $H_i(X_{\bullet})\in\add(M)$ for $0\leqslant i\leqslant k$ and  $H_i(X_{\bullet})=0$ for $i>k$ or $i<0$. Then $X_{\bullet} \simeq \bigoplus_{i=0}^{k} H_i(X_{\bullet})[i]$ in $\mathscr{D}^{\rm b}(\mathcal{A})$.
\end{Lem}

\begin{proof}
We take induction on $k$  to show Lemma \ref{finitesum}.
Clearly, Lemma \ref{finitesum} holds trivially for $k=0$.
Let $k>0$. Taking the canonical truncation of $X_\bullet$ at degree $k$ yields
a distinguished triangle in $\mathscr{D}^{\rm b}(\mathcal{A})$:
$$\tau_{\geqslant k}X_{\bullet}\longrightarrow X_{\bullet}\longrightarrow\tau_{\leqslant k-1}X_{\bullet}\lraf{f}(\tau_{\geqslant k}X_{\bullet})[1]$$
where $\tau_{\geqslant k}X_{\bullet}\simeq H_k(X_{\bullet})[k]\in\add(M[k])$ and
$H_i(\tau_{\leqslant k-1}X_{\bullet})\simeq H_i(X_{\bullet})$ for $i\leqslant k-1$ but
$H_i(\tau_{\leqslant k-1}X_{\bullet})=0$ for $i\geqslant k$.
In particular, $H_i(\tau_{\leqslant k-1}X_{\bullet})\in\add(M)$ for $0\leqslant i\leqslant k-1$.
Now, the induction on $k-1$ implies that
$$\tau_{\leqslant k-1}X_{\bullet}\simeq \bigoplus_{i=0}^{k-1} H_i(\tau_{\leqslant k-1}X_{\bullet})[i]\simeq \bigoplus_{i=0}^{k-1} H_i(X_{\bullet})[i]\in\add(\bigoplus_{i=0}^{k-1}M[i]).$$
Since $\Ext_\mathcal{A}^n(M, M) = 0$ for
$2 \leqslant n \leqslant k+1$, we have $\Hom_{\mathscr{D}^{\rm b}(\mathcal{A})}(\bigoplus_{i=0}^{k-1}M[i], M[k+1])=0$. This forces $\Hom_{\mathscr{D}^{\rm b}(\mathcal{A})}(\tau_{\leqslant k-1}X_{\bullet}, (\tau_{\geqslant k}X_{\bullet})[1])=0$. Thus $f=0$ and
$X_{\bullet}\simeq \tau_{\geqslant k}X_{\bullet}\oplus\tau_{\leqslant k-1}X_{\bullet} \simeq \bigoplus_{i=0}^{k} H_i(X_{\bullet})[i]$ in $\mathscr{D}^{\rm b}(\mathcal{A})$.
\end{proof}

\begin{Theo}\label{inf1}
Suppose that $M\in\mathcal{A}$ has a finite quasi-projective resolution $P_{\bullet}$.
If $\Ext_\mathcal{A}^n(M, M) = 0$ for $2\leqslant n\leqslant
{\rm hsup}(P_{\bullet})-{\rm hinf}(P_{\bullet})+1$, then
${\rm pd}_{\mathcal{A}}(M)<\infty$.
\end{Theo}
\begin{proof}
Let $t\coloneqq{\rm hsup}(P_{\bullet})$ and $s\coloneqq{\rm hinf}(P_{\bullet})$. Then $t$ and $s$ are finite with $s\leqslant t$. Since $P_{\bullet}$ is a finite quasi-projective resolution of $M$, we see that $P_{\bullet}\in\mathscr{C}^{\rm b}(\mathcal{P})$, $H_i(P_{\bullet})\simeq M^{a_i}$ for $s\leqslant i\leqslant t$ and  $H_i(P_{\bullet})=0$ for $i>t$ or $i<s$, where $a_i$ are nonnegative integers and not all zero. Suppose that $\Hom_{\mathscr{D}^{-}(\mathcal{A})}(M,M[n])=0$ for $2\leqslant n\leqslant
t-s+1$. By Lemma \ref{finitesum},
$P_{\bullet} \simeq \bigoplus_{i=s}^t H_i(P_{\bullet})[i]\simeq \bigoplus_{i=s}^t M^{a_i}[i]$ in $\mathscr{D}^{-}(\mathcal{A})$. Now, let $Q_{\bullet}$ be a deleted projective resolution of $M$. Then there is an isomorphism $P_{\bullet}\simeq\bigoplus_{i=s}^{t} Q_{\bullet}^{a_i}[i]$ in $\mathscr{D}^{-}(\mathcal{A})$, where both sides are bounded below complexes of projective objects. Since the localization functor $\mathscr{K}^{{-}}(\mathcal{P})\to \mathscr{D}^{-}(\mathcal{A})$ is a triangle equivalence, it follows that $P_{\bullet}\simeq\bigoplus_{i=s}^{t} Q_{\bullet}^{a_i}[i]$ in $\mathscr{K}^{-}(\mathcal{P})$. Note that $\mathscr{K}^{{\rm b}}(\mathcal{P})$ is closed under direct summands in $\mathscr{K}^{-}(\mathcal{P})$. Since $P_{\bullet}\in \mathscr{K}^{{\rm b}}(\mathcal{P})$ and at least one $a_i$ is not zero, we have $Q_{\bullet}\in\mathscr{K}^{{\rm b}}(\mathcal{P})$. This implies ${\rm pd}_{\mathcal{A}}(M)<\infty$.
\end{proof}

Theorem \ref{inf1} implies the following result.
\begin{Koro}\label{VP}
Let $M\in\mathcal{A}$. If $\qpd_{\mathcal{A}}(M)<\infty$ and $\Ext_\mathcal{A}^n(M,M)=0$ for $n\geqslant 2$, then ${\rm pd}_{\mathcal{A}}(M)<\infty$.
\end{Koro}

We present two direct corollaries of Corollary \ref{VP}.

\begin{Koro}\label{AR}
Let $R$ be a left coherent ring and let $M$ be a finitely presented $R$-module with finite quasi-projective dimension such that ${\rm Ext}_{R}^n(M,M\oplus R)=0$ for all $n>0$. Then $M$ is projective.
\end{Koro}

\begin{proof}
Let $\mathcal{A}\coloneqq R\modcat$. This is an abelian category with enough projective objects since $R$ is left coherent. By Corollary \ref{VP}, ${\rm pd}_R(M)<\infty$. Let $m\coloneqq{\rm pd}_R(M)$. We claim $m=0$.

In fact, since ${_R}M$ is finitely presented, it has a projective resolution $0\to P_m\stackrel{f_m}{\longrightarrow}P_{m-1}\to\cdots\to P_0\to M \to 0$ with all $P_i\in \prj{R}$ for $0\leqslant i\leqslant m$. If $m>0$, then the assumption ${\rm Ext}_{R}^m(M,R)=0$ implies ${\rm Ext}_{R}^m(M,P_m)=0$, and therefore $f_m$ is a split injection. This leads to $\pd_R(M)<m$, a contradiction. Thus $m=0$ and $M$ is projective.
\end{proof}

Recall that a ring $R$ is \emph{quasi-Frobenius} if all projective left $R$-modules are injective, or equivalently, all injective left $R$-modules are projective. Note that if $R$ is quasi-Frobenius, then it is left and right Noetherian and Artinian, and both $R\Modcat$ and $R\modcat$ are Frobenius abelian categories.

\begin{Koro}\label{tach}
Let $R$ be a quasi-Frobenius ring and let $M$ be an $R$-module with finite quasi-projective dimension. If ${\rm Ext}_{R}^n(M,M)=0$ for $n\geqslant 2$, then $M$ is projective.
\end{Koro}
\begin{proof}
By Corollary \ref{VP}, ${\rm pd}_R(M)<\infty$. Since $R$ is quasi-Frobenius, each $R$-module of finite projective dimension is projective. Thus $M$ is projective.
\end{proof}

Corollary \ref{VP} provides a method for constructing modules with infinite quasi-projective dimension.

\begin{Bsp}\label{Counterexample}
\rm{
Let $A$ be an algebra over a field $K$ presented by the quiver
$$
\begin{tikzcd}
1 \arrow[r, "\alpha"] & 2 \arrow["\beta", loop, distance=2em, in=325, out=35]
\end{tikzcd}
$$
with relations:$\beta\alpha=0=\beta^2$, where the composition of arrows is always taken from right to left. Let $S_1$ and $S_2$ be simple $A$-modules corresponding to the vertices $1$ and $2$, respectively. It can be checked that $\Omega(S_1)=S_2$ and $\Omega(S_2)=S_2$.
This implies ${\rm pd}_{A}(S_1)=\infty={\rm pd}_{A}(S_2)$ and $\Ext_{A}^i(S_1,S_1)=0$ for $i>0$ (in fact, $S_1$ is injective). Thus $\qpd_{A}(S_1)=\infty$ by {\rm Corollary \ref{VP}}. Moreover, since $S_2$ is periodic, $\qpd_{A}(S_2)=0$ by {\rm Theorem \ref{QPP}(3)}.

Note that there exists an exact sequence $0\to S_2\to P_1\to S_1\to 0$ of $A$-modules, where $P_1$ is the projective cover of $S_1$. Obviously, the inequality $\qpd_{A}(S_1)\leqslant \qpd_{A}(S_2)+1$ does not hold.
But we always have $\pd_{A}(M)\leqslant \pd_{A}(\Omega(M))+1$ for any $M\in\mathcal{A}$. This also shows that quasi-projective dimension is different from projective dimension.
}
\end{Bsp}

Next, we provide a sufficient condition for the inequality $\qpd_{\mathcal{A}}(M)\leqslant\qpd_{\mathcal{A}}(\Omega(M))+1$ to hold.

\begin{Prop}\label{ext}
Let $0\to N\to E\to M\to 0$ be an exact sequence in ${\mathcal{A}}$ such that $E$ is projective-injective. Then $\qpd_{\mathcal{A}}(M)\leqslant\qpd_{\mathcal{A}}(N)+1$.
\end{Prop}

\begin{proof}

The inequality in Proposition \ref{ext} automatically holds if $\qpd_{\mathcal{A}}(N)=\infty$. So, we assume $\qpd_{\mathcal{A}}(N)=m<\infty$.
Then $N$ admits a quasi-projective resolution of the form
$$
\begin{tikzcd}
P_{\bullet}:\quad0 \arrow[r] & P_m \arrow[r, "d_m"] & P_{m-1} \arrow[r, "d_{m-1}"] & \cdots \arrow[r, "d_1"] & P_0 \arrow[r, "d_0"] & \cdots \arrow[r, "d_{-r+1}"] & P_{-r} \arrow[r, "d_{-r}"] & 0,
\end{tikzcd}
$$
where $P_m\neq 0$, ${\rm hsup}(P_{\bullet})=0$, $H_{-r}(P_{\bullet})\not=0$, and for each $i\in\mathbb{Z}$,
there is an isomorphsim $g_i: H_i(P_{\bullet})\to N^{a_i}$ for some integers $a_i\in\mathbb{N}$ (not all zero). Let $0\to N\lraf{f} E\lra M\to 0$ be the exact sequence in Proposition \ref{ext}.
Since $E$ is injective, the composition
$$
\Ker(d_i)\twoheadrightarrow H_{i}(P_{\bullet})\lraf{g_i} N^{a_i}\lraf{f^{a_i}}E^{a_i}
$$ can be lifted to a morphism $h_i:P_i\longrightarrow E^{a_i}$ in $\mathcal{A}$. This is illustrated by the following commutative diagram
$$
\begin{tikzcd}
&P_{i+1}\arrow[dl, "d_{i+1}"'] \arrow[rr, "d_{i+1}"] &                                       & P_i \arrow[d, "h_i"] \\
{\Img(d_{i+1})}  \arrow[r, hook] &{\Ker(d_i)} \arrow[r] \arrow[rru, hook]         & H_i(P_{\bullet})\lraf{g_i} N^{a_i} \arrow[r, "f^{a_i}"] & E^{a_i}.
\end{tikzcd}
$$Clearly, $h_id_{i+1}=0$. Let
$$
E_{\bullet}\coloneqq0\longrightarrow E^{a_m}\lraf{0}\cdots\lraf{0} E^{a_0}\lra\cdots\lraf{0} E^{a_{-r}}\longrightarrow 0
$$
be a bounded complex over $\mathcal{A}$ with all differentials zero.
Since $E$ is projective, we have $E_\bullet\in\mathscr{C}^{\rm b}(\mathcal{P})$.
Set $h_{\bullet}\coloneqq(h_i)_{i\in\mathbb{Z}}: P_{\bullet}\lra E_{\bullet}$. Then
$h_{\bullet}$ is a chain map in $\mathscr{C}^{\rm b}(\mathcal{P})$. Now, we extend the chain map to
a distinguished triangle $P_{\bullet}\lraf{h_{\bullet}} E_{\bullet}\lra\Cone(h_{\bullet})\lra P_{\bullet}[1]$ in $\mathscr{K}^{\rm b}(\mathcal{P})$, where $\Cone(h_{\bullet})$ denotes the mapping
cone of $h_\bullet$.
By \cite[Corollary 10.1.4]{1994Weibel}, there is a long exact sequence
$$
\begin{tikzcd}
\cdots \arrow[r] & H_i(P_{\bullet}) \arrow[r, "H_i(h_{\bullet})"] & H_i(E_{\bullet}) \arrow[r] & H_i(\Cone(h_{\bullet})) \arrow[r] & H_{i-1}(P_{\bullet}) \arrow[r, "H_{i-1}(h_{\bullet})"] & \cdots
\end{tikzcd}.
$$
Clearly, $H_i(h_{\bullet})=f^{a_i}g_i$. Since $f$ is a monomorphism and $g_i$ is an isomorphsim, we obtain a short exact sequence
$$
0\longrightarrow H_i(P_{\bullet})\lraf{H_{i}(h_{\bullet})}H_i(E_{\bullet})\longrightarrow H_i(\Cone(h_{\bullet}))\longrightarrow 0.
$$
It follows from $\Coker(f)\simeq M$ that $H_i(\Cone(h_{\bullet}))\simeq M^{a_i}$. Thus $\Cone(h_{\bullet})$ is a finite quasi-projective resolution of $M$. By ${\rm hsup}(P_{\bullet})=0$, we have $a_i=0$ for $i>0$ and $a_0\neq 0$. This forces ${\rm hsup}(\Cone(h_\bullet))=0$. Note that $\sup(\Cone(h_{\bullet}))=m+1$ beacuse the $(m+1)$-th term of the complex $\Cone(h_\bullet)$ is exactly $P_m$. Therefore $\qpd_{\mathcal{A}}(M)\leqslant \sup(\Cone(h_\bullet))-{\rm hsup}(\Cone(h_\bullet))=m+1=\qpd_{\mathcal{A}}(N)+1.$
\end{proof}

\begin{Koro}\label{summand}
Let $\mathcal{A}$ be a Frobenius abelian category . For any $M\in\mathcal{A}$, $P\in\mathcal{P}$ and $i\in\mathbb{Z}$, $$\qpd_\mathcal{A}(M\oplus P)=\qpd_\mathcal{A}(M)=\qpd_\mathcal{A}(\Omega^i(M)).$$
\end{Koro}
\begin{proof}
By Theorem \ref{QPP}(1), $\qpd_\mathcal{A}(M\oplus P)\leqslant\qpd_\mathcal{A}(M)$. By Proposition \ref{ext}, $\qpd_\mathcal{A}(M)\leqslant\qpd_\mathcal{A}(\Omega(M))+1$. Since $\Omega(M)=\Omega(M\oplus P)$, we have $\qpd_\mathcal{A}(\Omega(M))=\qpd_\mathcal{A}(\Omega(M\oplus P))\leqslant\qpd_\mathcal{A}(M\oplus P)$ by Theorem \ref{QPP}$(2)$. Thus $\qpd_\mathcal{A}(M\oplus P)<\infty$ if and only if $\qpd_\mathcal{A}(M)<\infty$ if and only if $\qpd_\mathcal{A}(\Omega(M))<\infty$.  By Corollary \ref{FB}, the objects $M\oplus P$, $M$ and $\Omega(M)$ (and thus also $\Omega^i(M)$) have the same quasi-projective dimension.
\end{proof}

\textbf{Proof of Theorem \ref{main1.1}.} The statements $(1)$-$(3)$ are Proposition \ref{qp=p1}, Corollary \ref{VP} and Proposition \ref{ext}, respectively. It remains to show $(4)$.

For each $s\geqslant 0$, we denote by $P_{\bullet}^{\leqslant s}$ the brutal truncation at degree $s$, that is, $$P_{i}^{\leqslant s}=
\begin{cases}
P_i& \text{if $i\leqslant s$};\\
0 & \text{if $i>s$}.
\end{cases}$$ The monomorphism $\Omega^{n+m}(M)\longrightarrow P_{n+m-1}$ induces a distinguished triangle in $\mathscr{D}^{-}(\mathcal{A})$: $$\Omega^{n+m}(M)[n+m-1]\longrightarrow P_{\bullet}^{\leqslant n+m-1}\longrightarrow M\longrightarrow\Omega^{n+m}(M)[n+m],$$
where $M$ is identified with the complex $0\to\Omega^{n+m}(M)\to P_{n+m-1}\to P_{n+m-2}\to\cdots\to P_1\to P_0\to 0$, up to isomorphism. Similarly, there is a distinguished triangle in $\mathscr{D}^{-}(\mathcal{A})$:
$$
\Omega^{m}(M)[m-1]\longrightarrow P_{\bullet}^{\leqslant m-1}\longrightarrow M\longrightarrow\Omega^{m}(M)[m].
$$
Shifting this triangle by $n$ steps yields a new distinguished triangle $$\Omega^{m}(M)[n+m-1]\longrightarrow P_{\bullet}^{\leqslant m-1}[n]\longrightarrow M[n]\longrightarrow\Omega^{m}(M)[n+m].$$
Further, since the chain map $g_{\bullet}:P_{\bullet}\to P_{\bullet}[n]$ induces an isomorphism $\Omega^{n+m}(M)\to\Omega^m(M)$, we can construct the following commutative diagram  in $\mathscr{D}^{-}(\mathcal{A})$:
$$
\begin{tikzcd}
{\Omega^{n+m}(M)[n+m-1]} \arrow[d, "\simeq"'] \arrow[r] & P_{\bullet}^{\leqslant n+m-1} \arrow[d, "g_{\bullet}^{\leqslant n+m-1}"] \arrow[r]               & M \arrow[r] \arrow[d]                          & {\Omega^{n+m}(M)[n+m]} \arrow[d, "\simeq"'] \\
{\Omega^{m}(M)[n+m-1]} \arrow[r]                        & {P_{\bullet}^{\leqslant m-1}[n]} \arrow[r] \arrow[d]                                             & {M[n]} \arrow[r] \arrow[d]                     & {\Omega^{m}(M)[n+m]}             \\
                                                        & \Cone(g_{\bullet}^{\leqslant n+m-1}) \arrow[d] \arrow[r, no head, shift left] \arrow[r, no head] & \Cone(g_{\bullet}^{\leqslant n+m-1}) \arrow[d] &                                  \\
                                                        & {P_{\bullet}^{\leqslant n+m-1}[1]} \arrow[r]                                                     & {M[1],}                                         &
\end{tikzcd}
$$ where all columns and rows are distinguished triangles by the Octahedral axiom.
Clearly, $\Cone(g_{\bullet}^{\leqslant n+m-1})$ lies in $\mathscr{C}^{\rm b}(\mathcal{P})$. By $n\geqslant 2$, we see that $H_i(\Cone(g_{\bullet}^{\leqslant n+m-1}))\simeq M$ if $i=1$ or $n$, and $H_i(\Cone(g_{\bullet}^{\leqslant n+m-1}) )=0$ for other $i$. It follows that $\Cone(g_{\bullet}^{\leqslant n+m-1})$ is a finite quasi-projective resolution of $M$. Thus
$\qpd_{\mathcal{A}}(M)\leqslant{\rm sup}(\Cone(g_{\bullet}^{\leqslant n+m-1})) -{\rm hsup}(\Cone(g_{\bullet}^{\leqslant n+m-1}))=n+m-n=m.$ \hfill $\square$

\begin{Bsp}\label{newcase}
{\rm
let $A$ be an algebra over a field $K$ presented by the quiver
$$\begin{tikzcd}
1 \arrow[r, "\alpha"] & 2 \arrow[d, "\beta"]  \\
4 \arrow[u, "\delta"] & 3 \arrow[l, "\gamma"]
\end{tikzcd}$$
with relations: $\delta\gamma\beta\alpha=\beta\alpha\delta\gamma=0$. Denote by $J$ the Jacobson radical of $A$, and by $e_i$ the primitive, idempotent element corresponding to the vertex $i$ for $1\leqslant i\leqslant 4$. Then there exists an obvious algebra automorphism $\Phi: A\to A$ sending $e_1\mapsto e_3$, $e_2\mapsto e_4$, $e_3\mapsto e_1$ and $e_4\mapsto e_2$.

We will show that

{\rm $(1)$} $\qpd_{A}(Ae_{1}/Je_{1})=\qpd_{A}(Ae_{3}/Je_{3})=\qpd_{A}(Ae_{2}/J^{2}e_{2})=\qpd_{A}(Ae_{4}/J^{2}e_{4})=2$.

{\rm $(2)$} $\qpd_{A}(Ae_{2}/J^{3}e_{2})=\qpd_{A}(Ae_{4}/J^{3}e_{4})=1$.

\medskip

In fact, $A$ is a Nakayama algebra, and thus is of finite representation type. By {\rm \cite[Theorem V.3.5]{2006ASS}}, every indecomposable $A$-module is isomorphic to $P/J^tP$ for some indecomposable projective $A$-module $P$ and for some integer $t$ with $1\leqslant t\leqslant 5$. Moreover, the indecomposable projective modules of $A$ are determined by their composition sequences:
\begin{center}
$Ae_1=\begin{matrix}
1 \\ 2 \\ 3 \\ 4
\end{matrix},
\qquad
Ae_2=\begin{matrix}
2\\ 3 \\ 4 \\ 1\\ 2
\end{matrix},
\qquad
Ae_3=\begin{matrix}
3 \\ 4 \\ 1 \\ 2
\end{matrix},
\qquad
Ae_4=\begin{matrix}
4\\ 1 \\ 2 \\ 3\\ 4
\end{matrix}.
$
\end{center}Clearly, $Ae_2$ and $Ae_4$ are projective-injective, and indecomposable modules of projective dimension $1$ are only simple $A$-modules $Ae_2/Je_2$ and $Ae_4/Je_4$.
Thus, for any $X\in A\modcat$, $\pd_{A}(X) \leqslant 1$ if and only if
$X \in \add(A \oplus(Ae_2/Je_2)\oplus (Ae_4/Je_4))$.
This implies that if $\pd_A(X) \leqslant 1$, then the socle $\soc(X)$ of ${_A}X$ is a direct sum of finitely many copies of $Ae_2/Je_2$ and $Ae_4/Je_4$. Let
$$X\in\{A e_{1}/Je_{1},\; Ae_{3}/Je_{3},\; Ae_{2}/J^2e_{2},\; Ae_{4}/ J^2e_{4}\}.$$
Then $\soc(X)$ is a direct sum of finitely many copies of simple $A$-modules $Ae_1/Je_1$ and $Ae_3/Je_3$. It follows that $X$ can't be embedded into any $A$-module of projective dimension at most $1$.
By {\rm Corollary \ref{general 3.4}}(2), $\qpd_A(X)\geqslant 2$. To show $\qpd_A(X)\leqslant 2$, we
construct the following two chain maps of complexes of $A$-modules:
$$
\begin{tikzcd}
\cdots \arrow[r, "\cdot\delta\gamma"]  & Ae_3 \arrow[r, "\cdot\beta\alpha"] \arrow[d, no head, shift right] \arrow[d, no head] & Ae_1 \arrow[r, "\cdot\delta\gamma"] \arrow[d, no head, shift right] \arrow[d, no head] & Ae_3 \arrow[r, "\cdot\beta\alpha"] \arrow[d, "\cdot\beta"] & Ae_1 \arrow[r, "\cdot \delta\gamma\beta"] \arrow[d, no head, shift right] \arrow[d, no head] & Ae_2 \arrow[r, "\cdot \alpha"] \arrow[d] & Ae_1 \arrow[r] & 0 \\
\cdots \arrow[r, "\cdot\delta\gamma"'] & Ae_3 \arrow[r, "\cdot\beta\alpha"']                                                   & Ae_1 \arrow[r, "\cdot \delta\gamma\beta"']                                             & Ae_2 \arrow[r, "\cdot \alpha"']                            & Ae_1 \arrow[r]                                                                               & 0,                                        &                &
\end{tikzcd}
$$
where each row is a deleted projective resolution of $Ae_{1}/Je_{1}$;
$$
\begin{tikzcd}
\cdots \arrow[r, "\cdot\beta\alpha"]  & Ae_{1} \arrow[r, "\cdot\delta\gamma"] \arrow[d, no head, shift right] \arrow[d, no head] & Ae_{3} \arrow[r, "\cdot\beta\alpha"] \arrow[d, no head, shift right] \arrow[d, no head] & Ae_{1} \arrow[r, "\cdot\delta\gamma"] \arrow[d, "\cdot\delta"] & Ae_{3} \arrow[r, "\cdot\beta\alpha\delta"] \arrow[d, "\cdot\beta"] & Ae_{4} \arrow[r, "\cdot \gamma\beta"] \arrow[d] & Ae_{2} \arrow[r] & 0 \\
\cdots \arrow[r, "\cdot\beta\alpha"'] & Ae_{1} \arrow[r, "\cdot\delta\gamma"']                                                   & Ae_{3} \arrow[r, "\cdot\beta\alpha\delta"']                                             & Ae_{4} \arrow[r, "\cdot \gamma\beta"']                         & Ae_{2} \arrow[r]                                                   & 0,                                               &                  &
\end{tikzcd}
$$ where each row is a deleted projective resolution of $Ae_{2}/J^2e_{2}$.
By {\rm Theorem \ref{main1.1}(4)}, $\qpd_A\,(Ae_1/Je_1)\leqslant2$ and $\qpd_{A}(Ae_{2}/J^2e_{2})\leqslant 2$. Thus $\qpd_{A}\,(Ae_1/Je_1)=\qpd_{A}(Ae_{2}/J^2e_{2})=2$. By the isomorphism $\Phi$, we have $\qpd_A(Ae_{3}/Je_{3})=2=\qpd_A(Ae_{4}/J^2e_{4})$.

Finally, we consider the following two exact sequences:
$$0 \to Ae_1/J^2e_1 \to Ae_2 \to Ae_2/J^3e_2 \to 0,\quad 0 \to Ae_2/J^3e_2 \to Ae_4 \to Ae_4/J^2e_4 \to 0.$$ Since $Ae_1/J^2e_1$ is periodic, it follows from
{\rm Theorem \ref{QPP}(3)} that $\qpd_{A}(Ae_1/J^2e_1)=0$.
Since $Ae_2$ is projective-injective, $\qpd_{A}(Ae_2/J^3e_2) \leqslant 1$ by {\rm Theorem \ref{main1.1}(3)}. If $\qpd_{A}(Ae_2/J^3e_2)=0$, then {\rm Theorem \ref{main1.1}(3)} implies $\qpd_{A}(Ae_{4}/J^2e_{4})\leqslant 1$, which is contradictory to $\qpd_{A}(Ae_{4}/J^2e_{4})=2$. Thus $\qpd_A(Ae_2/J^3e_2)=1$. Similarly, $\qpd_A(Ae_{4}/J^3e_{4})=1$ by the isomorphism $\Phi$.

It is easy to see that all the $A$-modules listed in $(1)$ and $(2)$ have infinite projective dimension.
}
\end{Bsp}

\section{Quasi-global dimensions of rings}\label{4}
In this section, we introduce the notion of quasi-global dimension for left Noetherian rings by analogy with the classical global dimension. In Section \ref{4.1}, we investigate some basic properties of quasi-global dimension. For instance, quasi-global dimension is equal to finitistic dimension whenever the former is finite (Proposition \ref{basic}(1)), and both stable equivalence of Morita type and derived equivalence preserve quasi-global dimension of self-injective algebras (Corollary \ref{moritatype}). In Section \ref{4.2}, we calculate the quasi-global dimensions for a class of Nakayama algebras (Theorem \ref{main}). It turns out that these algebras have finite quasi-global dimension.

\subsection{Left Noetherian rings}\label{4.1}
Throughout this section, let $R$ be a \emph{left Noetherian} ring. Recall that the \emph{global} and \emph{finitistic dimensions} of $R$, denoted by $\gd(R)$ and $\fd(R)$, respectively, are defined as
$$\gd(R)\coloneqq\sup\{\pd_R\,M\,|\,M\in R\text{-}{\rm mod}\}\;\;\mbox{and}\;\;\fd(R)\coloneqq\sup\{\pd_{R}\,M\,|\,M\in R\text{-}{\rm mod}\,\,\mbox{with}\,\,\pd_{R}(M)<\infty\}.$$
Clearly, $\fd(R)\leqslant \gd(R)$. Similarly, we can introduce the quasi-global dimension of $R$ as follows.

\begin{Def}
The {\rm quasi-global dimension} of $R$, denoted by $\qgd(R)$, is defined as $$\qgd(R)\coloneqq\sup\{\qpd_R(M)\,|\,M\in R\text{-}{\rm mod}\}.$$
\end{Def}
Since $\qpd_R(M)\leqslant \pd_R(M)$ for any $R$-module $M$, we have $\qgd(R)\leqslant \gd(R)$.
The inequality may be strict (for example, see Theorem \ref{main}). However, if $\gd(R)<\infty$, then
$\qgd(R)=\gd(R)$ by Theorem \ref{main1.1}$(1)$.

We first collect some basic properties of quasi-global dimension.

\begin{Prop}\label{basic}
Let $R$ and $S$ be left Noetherian rings. Then:

$(1)$ $\fd(R)=\sup\{\qpd_{R}\,M\,|\,M\in R\text{-}{\rm mod}\,\,and\,\,\qpd_{R}\,M<\infty\}$.  In particular, if $\qgd(R)<\infty$, then $\fd(R)=\qgd(R)$.

$(2)$ If $R$ and $S$ are Morita equivalent, then $\qgd(R)=\qgd(S)$.

$(3)$ $\qgd(R\times S)={\rm max}\{\qgd(R),\,\qgd(S)\}$.

\noindent Furthermore, if $R$ and $S$ are finite-dimensional algebras over a field $K$, then

$(4)$  $\qgd(R\otimes_KS)\geqslant{\rm max}\{\qgd(R),\,\qgd(S)\}$.
\end{Prop}
\begin{proof}
$(1)$ follows from Theorem \ref{main1.1}$(1)$ and Proposition \ref{general 3}$(a)$.

$(2)$  Suppose that $R$ and $S$ are Morita equivalent. Then there exists an $S$-$R$-bimodule $P$ such that both ${_S}P$ and $P{_R}$ are finitely generated and projective and the tensor functor $P\otimes_R-: R\Modcat\to S\Modcat$ is an equivalence of abelian categories. Since $R$ and $S$ are left Noetherian, this equivalence can be restricted to an equivalence $F:R\modcat\to S\modcat$ of abelian categories. Let $M$ be a finitely generated $R$-module with a finite quasi-projective resolution $P_{\bullet}$. By Corollary \ref{general 3.4}(1), we may assume that $P_{\bullet}$ belongs to $\mathscr{K}^{{\rm b}}(R\text{-}{\rm proj})$. Since $F$ preserves projective modules and commutes with homology functors, it follows that $F(P_{\bullet})$ is a finite quasi-projective resolution of $F(M)$. Now, it is easy to show $\qgd(R)=\qgd(S)$.

$(3)$ Without loss of generality, assume $\qgd(R)\geqslant\qgd(S)$. We first show the equality $\qpd_R\,M=\qpd_{R\times S}\,(M\times S)$ for any $M\in R\modcat$. Let $P^M_{\bullet}$ be a quasi-projective resolution of $M$. By definition, $H_i(P^M_{\bullet})\simeq M^{a_i}$ for some $a_i\in\mathbb{N}$. Define
$S_{\bullet}\coloneqq\bigoplus_{i\in\mathbb{Z}}S^{a_i}[a_i]$.
Then $P^M_{\bullet}\times S_{\bullet}$ is a quasi-projective resolution of $M\times S$. This implies
$\qpd_{R\times S}\,(M\times S)\leqslant\qpd_R\,M$. The converse $\qpd_R\,M\leqslant\qpd_{R\times S}\,(M\times S)$ follows from the fact that any quasi-projective resolution of the $(R\times S)$-module $M\times S$ is of the form $Q^M_{\bullet}\times Q^S_{\bullet}$, where $Q^M_{\bullet}$ and $Q^S_{\bullet}$ are quasi-projective resolutions of $M$ and $S$, respectively. Thus $\qpd_R\,M=\qpd_{R\times S}\,(M\times S)$. It follows that $\qgd(R\times S)\geqslant{\rm max}\{\qgd(R),\,\qgd(S)\}$. In particular, if $\qgd(R)=\infty$, then $\qgd(R\times S)=\infty$.

It remains to show that $\qgd(R\times S)\leqslant{\rm max}\{\qgd(R),\,\qgd(S)\}$. Suppose $\qgd(R)<\infty$. Then $\qgd(S)<\infty$. It follows that finitely generated $R$-modules and finitely generated $S$-modules have finite quasi-projective resolutions. Let $M\times N$ be an arbitrary finitely generated $(R\times S)$-module. Then $M\in R\text{-}{\rm mod}$ and $N\in S\text{-}{\rm mod}$. We take $P_{\bullet}$ and $Q_{\bullet}$ to be any finite quasi-projective resolutions of $M$ and $N$, respectively.
Then $H_i(P_{\bullet})\simeq M^{c_i}$ and $H_i(Q_{\bullet})\simeq N^{d_i}$ for some $c_i$, $d_i\in\mathbb{N}$. Since $P_{\bullet}$ and $Q_{\bullet}$ are bounded complexes, both $c_i$ and $d_i$ are nonzero for finitely many $i$. Set $$F_{\bullet}\coloneqq\big(\bigoplus_{j\in\mathbb{Z}}(P_{\bullet}[j])^{d_j}\big)
\times\big(\bigoplus_{i\in\mathbb{Z}}(Q_{\bullet}[i])^{c_i}\big).$$ Then $F_{\bullet}$ is a finite quasi-projective resolution of $M\times N$. In fact, $H_{k}(F_{\bullet})=(M\times N)^{\sum_{i+j=k}c_id_j}$ for each $k\in\mathbb{Z}$. Note that ${\rm sup}\,F_{\bullet}={\rm max}\{{\rm sup}\,P_{\bullet}+{\rm hsup}\,Q_{\bullet},\,{\rm sup}\,Q_{\bullet}+{\rm hsup}\,P_{\bullet}\}$ and ${\rm hsup}\,F_{\bullet}={\rm hsup}\,P_{\bullet}+{\rm hsup}\,Q_{\bullet}.$ This implies $\qpd_{R\times S}\,(M\times N)\leqslant {\rm max}\{\qpd_R\,M,\qpd_S\,N\}$. Thus $\qgd(R\times S)\leqslant{\rm max}\{\qgd(R),\,\qgd(S)\}$.

$(4)$ The inequality holds trivially if $\qgd(R\otimes_KS)=\infty$. So, we assume  $\qgd(R\otimes_KS)<\infty$. Let $M\in R\modcat$. Then $M\otimes_KS\in(R\otimes_KS)\modcat$. Let $P_{\bullet}\in\Cb{\prj{(R\otimes_KS)}}$ be any finite quasi-projective resolution of $M\otimes_KS$ with $H_i(P_{\bullet})\simeq (M\otimes_KS)^{a_i}$ for each $i$, where $a_i\in\mathbb{N}$ are not all zero. Since $S$ is a finite-dimensional $K$-algebra, each $P_i$ as an $R$-module is finitely generated and projective, and $H_{i}(P_{\bullet})\simeq M^{a_i {\rm dim}_K(S)}$ as $R$-modules. It follows that $P_{\bullet}\in\Cb{\prj{R}}$ is a finite quasi-projective resolution of ${_R}M$. Thus   $\qpd_{R}(M)\leqslant\qgd_{R}(M\otimes_KS)$, which forces $\qgd(R)\leqslant\qgd(R\otimes_KS)$. Similarly, $\qgd(S)\leqslant\qgd(R\otimes_KS)$.
\end{proof}

\begin{Koro}\label{QDZO}
Let $R$ be a quasi-Frobenius ring. Then $\qgd(R)=0$ or $\infty$. If $\qgd(R)=0$, then every $R$-module $M$ with $\Ext_R^i(M,M)=0$ for $i\geqslant 2$ is projective.
\end{Koro}

\begin{proof}
Since $R$ is quasi-Frobenius, each $R$-module of finite projective dimension is projective. This implies $\fd(R)=0$. By Proposition \ref{basic}(1), $\qgd(R)=0$ or $\infty$. If $\qgd(R)=0$, then each $R$-module $M$ has quasi-projective dimension $0$. If, in addition, $\Ext_R^i(M,M)=0$ for all $i\geqslant 2$, then $M$ is projective by Corollary \ref{tach}.
\end{proof}

\begin{Rem}\label{TSC}
Let $R$ be a self-injective Artin algebra. Recall that \emph{Tachikawa’s second conjecture} asserts that a finitely generated $R$-module $M$ is projective if $\Ext_R^i(M,M)=0$ for $i\geqslant 1$. By Corollary \ref{QDZO}, Tachikawa’s second conjecture holds for a self-injective Artin algebra with finite quasi-global dimension (or equivalently, with quasi-global dimension $0$). Note that if $R$ is of \emph{finite representation type}, that is, there are only finitely many isomorphism classes of indecomposable, finitely generated $R$-modules, then $\qgd(R)=0$. In this case, each $M\in R\modcat$ is periodic and $\qpd_R(M)=0$ by Theorem \ref{QPP}(3).

It would be interesting to classify all self-injective Artin algebras with quasi-global dimension $0$.
\end{Rem}

\begin{Prop}\label{STE}
Let $A$ and $B$ be quasi-Frobenius rings. Suppose that $F:A\modcat\to B\modcat$ is an exact functor and induces an equivalence $A\text{-}\underline{{\rm mod}}\to B\text{-}\underline{{\rm mod}}$ of stable module categories. Then $\qgd(B)\leqslant\qgd(A)$.
\end{Prop}

\begin{proof}
Since $F$ is an exact functor and induces an equivalence $A\text{-}\underline{{\rm mod}}\to B\text{-}\underline{{\rm mod}}$, we see that $F$ commutes with homology functors, and preserves and detects finitely generated projective modules. It follows that if $P_{\bullet}\coloneqq(P_i)_{i\in\mathbb{Z}}$ is a quasi-projective resolution of an $A$-module $M$, then $F(P_\bullet)\coloneqq\big(F(P_i)_{i\in\mathbb{Z}}\big)$ is a quasi-projective resolution of the $B$-module $F(M)$. This implies $\qpd_B(F(M))\leqslant \qpd_A(M)$. Let $Y\in B\modcat$. As $F:A\text{-}\underline{{\rm mod}}\to B\text{-}\underline{{\rm mod}}$ is dense, there exists $X\in A\modcat$ and $P_1, P_2\in\prj{B}$ such that $Y\oplus P_1\simeq F(X)\oplus P_2$. Since $A$ and $B$ are quasi-Frobenius, $\qpd_B(Y)=\qpd_B(F(X))$ by Corollary \ref{summand}. Thus $\qpd_B(Y)\leqslant \qpd_A(X)$. This forces $\qgd(B)\leqslant\qgd(A)$.
\end{proof}

Next, we apply Proposition \ref{STE} to special stable equivalences between self-injective algebras. For the definitions and constructions of stable equivalences of Morita type and derived equivalences, we refer to the survey article \cite{xi18}.

\begin{Koro}\label{moritatype}
Let $A$ and $B$ be finite-dimensional self-injective algebras over a field $K$. If $A$ and $B$ are stably equivalent of Morita type or derived equivalent, then they have the same quasi-global dimension.
\end{Koro}

\begin{proof}
Clearly, finite-dimensional self-injective algebras are quasi-Frobenius rings. If $A$ and $B$ are stably equivalent of Morita type, then there exists an $A$-$B$-module $X$ and an $B$-$A$-module $Y$ such that the tensor functors $Y\otimes_A-: A\modcat\to B\modcat$ and $X\otimes_B-:B\modcat\to A\modcat$ are exact and induce equivalences $A\text{-}\underline{{\rm mod}}\simeq B\text{-}\underline{{\rm mod}}$ of stable module categories. In this case, $\qgd(B)=\qgd(A)$ by Proposition \ref{STE}. If $A$ and $B$ are derived equivalent, then they are stably equivalent of Morita type by \cite[Corollary 5.5]{1991Rickard} and thus $\qgd(B)=\qgd(A)$.
\end{proof}

Finally, we provide an example of a self-injective algebra with \emph{infinite} quasi-global dimension.

Let $K$ be a field with a non-zero element $q\in K$ that is not a root of unity. The local symmetric $K$-algebra $A$, defined by Liu–Schulz (see \cite{1994liu}), is generated by $x_0$, $x_1$, $x_2$ with the relations: $x_i^2=0$, and $x_{i+1}x_{i}+qx_{i}x_{i+1}=0$ for $i=0,1,2$.

\begin{Prop}
Let $A$ be Liu–Schulz algebra. Then $\qgd(A)=\infty$.
\end{Prop}
\begin{proof}

Since $A$ is a finite-dimensional symmetric $K$-algebra, every non-projective module has infinite projective dimension. By \cite[Lemma 6.6]{2012Chen}, there exists a finitely generated non-projective $A$-module $I_0$ such that $\Ext_A^i(I_0,I_0)=0$ for all $i\geqslant 2$. It follows from Theorem $\ref{main1.1}(2)$ that $\qpd_A(I_0)=\infty$. Thus $\qgd(A)=\infty$.
\end{proof}

\subsection{Nakayama algebras}\label{4.2}
In this subsection, we fix an arbitrary field $K$, an integer $n\geqslant 2$ and the following quiver \(Q\):

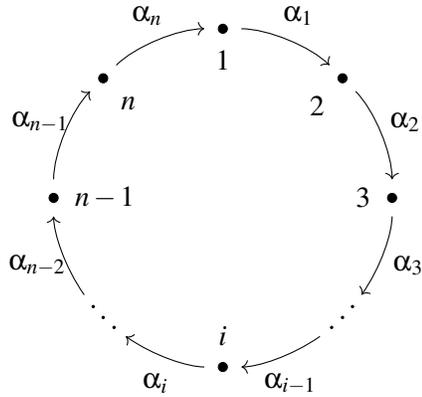
\[\begin{tikzpicture}[scale=0.75]
\foreach \ang\lab\anch in {90/1/north, 45/2/{north east}, 0/3/east, 270/i/south, 180/{n-1}/west, 135/n/{north west}}{
  \draw[fill=black] ($(0,0)+(\ang:3)$) circle (.08);
  \node[anchor=\anch] at ($(0,0)+(\ang:2.8)$) {$\lab$};
}

\foreach \ang\lab in {90/1,45/2,180/{n-1},135/n}{
  \draw[->,shorten <=7pt, shorten >=7pt] ($(0,0)+(\ang:3)$) arc (\ang:\ang-45:3);
  \node at ($(0,0)+(\ang-22.5:3.5)$) {$\boldsymbol\alpha_{\lab}$};
}

\draw[->,shorten <=7pt] ($(0,0)+(0:3)$) arc (360:325:3);
\draw[->,shorten >=7pt] ($(0,0)+(305:3)$) arc (305:270:3);
\draw[->,shorten <=7pt] ($(0,0)+(270:3)$) arc (270:235:3);
\draw[->,shorten >=7pt] ($(0,0)+(215:3)$) arc (215:180:3);
\node at ($(0,0)+(0-20:3.5)$) {$\boldsymbol\alpha_3$};
\node at ($(0,0)+(315-25:3.5)$) {$\boldsymbol\alpha_{i-1}$};
\node at ($(0,0)+(270-20:3.5)$) {$\boldsymbol\alpha_i$};
\node at ($(0,0)+(225-25:3.5)$) {$\boldsymbol\alpha_{n-2}$};

\foreach \ang in {310,315,320,220,225,230}{
 \draw[fill=black] ($(0,0)+(\ang:3)$) circle (.02);
}
\end{tikzpicture}.\]
Denote by
\[\rho_i=
\begin{cases}
\alpha_{n}\cdots\alpha_1& \text{for $i=1,$}\\
\alpha_{i-1}\cdots\alpha_1\alpha_n\cdots\alpha_{i}& \text{for $2\leqslant i\leqslant n,$}
\end{cases}\]
the composition of certain arrows in \(Q\). Let $m$ and $\lambda_j$ be integers with $1 \leqslant m \leqslant n$, $1\leqslant j\leqslant m$ and
$1 \leqslant \lambda_1 < \lambda_2 < \dots < \lambda_m \leqslant n$. We define
\[
A_{n,m} \coloneqq KQ / \langle \rho_{\lambda_1}, \rho_{\lambda_2}, \dots, \rho_{\lambda_m} \rangle,
\] the quotient algebra of the path algebra $KQ$ modulo the ideal generated by $\rho_{\lambda_i}$ for all $1\leqslant i\leqslant m$. Clearly, $A_{n,m}$ is a Nakayama algebra.
The main result of this subsection is to compare its quasi-global dimension with its global dimension.

\begin{Theo}\label{main}
$(1)$ $\gd(A_{n,1})=\qgd(A_{n,1})=2$.

$(2)$ $\gd(A_{n,n})=\infty$ and $\qgd(A_{n,n})=0$.

$(3)$ If $n>2$ and $1<m<n$, then $\gd(A_{n,m})=\infty$ and $\qgd(A_{n,m})=2$.
\end{Theo}

To show Theorem \ref{main}, we introduce some notation and establish several lemmas.

Let $J$ be the Jacobson radical of the algebra $A_{n,m}$. The Loewy length of an $A_{n,m}$-module $M$ is denote by $$\ell(M)\coloneqq\min\{s\in\mathbb{N}^{+}\,|\,J^sM=0\}.$$ Let $S_i$ and $P_i$ be the simple and indecomposable projective $A_{n,m}$-module corresponding to each vertex $i\in Q_0$, respectively. For simplicity, we denote by $L(i,k)$ the indecomposable $A_{n,m}$-module that has top $S_i$ and Loewy length $k$. In particular, $L(i,1)=S_i$. Moreover, we set $\Delta\coloneqq\{\lambda_1,\cdots\lambda_m\}$ and
define a bijection $\sigma:\{1,\cdots,n\}\rightarrow\{1,\cdots,n\}$ by \[\sigma(i)=
\begin{cases}
i+1 & \text{for $1\leqslant i\leqslant n-1$,}\\
1& \text{for $i=n.$}
\end{cases}\]
For $r\geqslant 1$, we denote the compositions $\underbrace{\sigma  \cdots  \sigma}_{r \text{ times}}$ and $\underbrace{\sigma^{-1}  \cdots  \sigma^{-1}}_{r \text{ times}}$ by $\sigma^r$ and $\sigma^{-r}$, respectively. As usual, $\sigma^0$ stands for the identity map of
$\{1,\cdots,n\}$. In the following, we always fix $i\in\{1, \dots, n\}$.

\begin{Lem}\label{basic2}
The following statements are true.

$(1)$ For any $A_{n,m}$-module $L(i,k)$, we have $0<k<2n$.

$(2)$ $\ell(P_i)=n + \inf\{r\in\mathbb{N}\mid\sigma^r(i) \in \Delta\}$. In particular,
$ n\leqslant \ell(P_i)<2n$.

$(3)$ If $n \leqslant k <\ell(P_{i})$, then $\ell(P_{\sigma^{k}(i)}) = \ell(P_{i}) - (k - n)$.

$(4)$ If $0<k < \ell(P_{i})$, then $\Omega(L(i,k)) \simeq L\big(\sigma^{k}(i),\ell(P_{i})-k\big)$.
\end{Lem}

\begin{proof}
$(1)$ Clearly, $k > 0$. Conversely, assume $k \geqslant 2n$. Since $L(i,k)$ is a quotient of $P_i$, we have $\ell(P_{i}) \geqslant k\geqslant 2n$. This implies that the path
$p_i:=\alpha_{\sigma^{n-2}(i)} \cdots\alpha_{\sigma(i)}\alpha_{i} \alpha_{\sigma^{n-1}(i)}\cdots  \alpha_{\sigma(i)} \alpha_i$ is nonzero in $A_{n,m}$. In the quiver $Q$, any path of length $n$ (for example, $\rho_j$ for $1\leqslant j\leqslant n$) is a subpath of $p_i$. Thus $\rho_j \neq 0$ in $A_{n,m}$, which is a contradiction to the definition of $A_{n,m}$.

$(2)$ By definition of $A_{n,m}$, we have $\ell(P_{i}) \geqslant n$. Clearly, $\ell(P_{i}) = n$ if and only if $i \in \Delta$. Suppose $\ell(P_{i}) = n + r$ for some $r > 0$. Then
\[\alpha_{\sigma^{n+r-2}(i)}  \cdots  \alpha_{\sigma(i)}  \alpha_i \neq 0\quad\text{and}\quad\alpha_{\sigma^{n+r-1}(i)}\alpha_{\sigma^{n+r-2}(i)}  \cdots  \alpha_{\sigma(i)}  \alpha_i =0.\]This implies \[\alpha_{\sigma^{n+r-2}(i)}  \cdots  \alpha_{\sigma^r(i)} \neq 0\quad\text{and}\quad\alpha_{\sigma^{n+r-1}(i)}\alpha_{\sigma^{n+r-2}(i)}  \cdots  \alpha_{\sigma^r(i)}=0.\]
Thus $\sigma^r(i)\in\Delta$.

We now prove that if there exists an integer $t>0$ such that $\sigma^{t}(i) \in \Delta$, then $t \geqslant r$.

In fact, if $t< r$, then $\alpha_{\sigma^{n+r-2}(i)}  \cdots   \alpha_{\sigma^r(i)}\cdots\alpha_{\sigma^t(i)} \neq 0$, and therefore $n+r-2-t+1=n+(r-t)-1\geqslant n$, which contradicts  $\sigma^t(i)\in\Delta$.

$(3)$ Assume $\ell(P_{i}) = n + r$ with $r\in\mathbb{N^{+}}$. Then $\alpha_{\sigma^{n+r-2}(i)} \cdots  \alpha_{\sigma(i)}  \alpha_i\neq 0$. By $(1)$, $r < n$. By $(2)$, $\sigma^r(i) \in \Delta$. It follows that \[\alpha_{\sigma^{n+r-1}(i)}\alpha_{\sigma^{n+r-2}(i)}  \cdots\alpha_{\sigma^n(i)}\cdots \alpha_{\sigma^r(i)} =0.\]
The case $k = n$ is clear. Now assume $k = n + s$ with $s>0$. Then $s<r$. This implies
\[\alpha_{\sigma^{n+r-2}(i)}  \cdots  \alpha_{\sigma^r(i)} \cdots \alpha_{\sigma^{s}(i)}\neq 0\quad\text{and}\quad\alpha_{\sigma^{n+r-1}(i)}\alpha_{\sigma^{n+r-2}(i)}  \cdots\alpha_{\sigma^r(i)}\cdots \alpha_{\sigma^s(i)} =0.\]
Thus $(3)$ holds.

$(4)$ Note that $\Omega(L(i,k))=J^{k}e_{i}$. Let \(f \colon P_{\sigma^{k}(i)} \to J^{k}e_i\) be the \(A_{n,m}\)-module homomorphism induced by $e_{\sigma^k(i)}\mapsto\alpha_{\sigma^{k-1}(i)}\alpha_{\sigma^{k-2}(i)}\cdots\alpha_i$. Then \(\Ker f = J^{\ell(P_i) - k}e_{\sigma^{k}(i)}\) and thus $(4)$ holds.
\end{proof}

\begin{Lem}\label{pdfinite}
Let $n \leqslant k <\ell(P_{i})$. Then $\qpd_{A_{n,m}}(L(i,k))=\pd_{A_{n,m}}(L(i,k))=2$.
\end{Lem}
\begin{proof}
 Let $\ell(P_{i}) = n + r$ with $r>0$. By Lemma \ref{basic2}(2), \( i \notin \Delta \). Since $k <\ell(P_{i})$, we see from Lemma \ref{basic2}(4) that $\Omega(L(i,k))\simeq L(\sigma^{k}(i),n+r-k)$. By  $k\geqslant n$ and Lemma \ref{basic2}(3), $\ell(P_{\sigma^{k}(i)})=n+r-(k-n)=2n+r-k$. Since $2n+r-k>n+r-k$, it is clear that $L(\sigma^{k}(i),n+r-k)$ is non-projective. Thus $\Omega(L(i,k))$ is non-projective.
 Now, there is a series of isomorphisms:
 \[\Omega^2(L(i,k))\simeq \Omega(L(\sigma^{k}(i),n+r-k))\simeq L(\sigma^{n+r}(i),(2n+r-k)-(n+r-k))=L(\sigma^{n+r}(i),n).
 \] Note that $\sigma^{n+r}(i)=\sigma^r(i)$. By Lemma \ref{basic2}(2), $\sigma^r(i)\in\Delta$, which forces $\ell(P_{\sigma^{n+r}(i)})=n$. This implies that  $\Omega^2(L(i,k))$ is projective, and therefore
 $\pd_{A_{n,m}}(L(i,k))=2$.  By Proposition \ref{qp=p1}, $\qpd_{A_{n,m}}(L(i,k))=\pd_{A_{n,m}}(L(i,k))$.
\end{proof}

\begin{Lem}\label{Calculation}
Let $0 < k < n$. Then

\[\qpd_{A_{n,m}}(L(i,k))\leqslant
\begin{cases}
0& \text{if $\,i,\,\sigma^k(i)\in\Delta$;}\\
1& \text{if $\,i\notin\Delta,\, \sigma^k(i)\in\Delta$;}\\
2& \text{if $\,\sigma^k(i)\notin\Delta$.}
\end{cases}\]
\end{Lem}

\begin{proof}
Since $0<k<n$, we have $\Omega(L(i,k))\simeq L\big(\sigma^{k}(i),\ell(P_{i})-k\big)$ by Lemma \ref{basic2}(2)(4). In the following, we divide the proof of Lemma \ref{Calculation} into four cases.

\textbf{Case 1.} $i,\,\sigma^{k}(i)\in\Delta$.

By Lemma \ref{basic2}(2), $\ell(P_{i})=\ell(P_{\sigma^{k}(i)})=n$ and $L\big(\sigma^{k}(i),\ell(P_{i})-k\big)=L(\sigma^k(i),n-k)$ which is non-projective. It follows from
Lemma \ref{basic2}(4) that \[\Omega^{2}(L(i,k))\simeq\Omega(L(\sigma^k(i),n-k))\simeq L(\sigma^{n}(i),n-(n-k))=L(i,k).\]
Thus $L(i,k)$ is periodic. By Theorem \ref{QPP}(3),  $\qpd_{A_{n,m}}(L(i,k))=0$.

In the later proof, we will use the following (minimal) projective resolution of the periodic module $L(i,k)$ for several times:
\[\begin{tikzcd}
\cdots \arrow[r] & P_{\sigma^{k}(i)} \arrow[r, "g"] & P_{i} \arrow[r, "h"] & P_{\sigma^{k}(i)} \arrow[r, "g"] & P_{i} \arrow[r] & {L(i,k)} \arrow[r] & 0\tag*{$(*)$}
\end{tikzcd}\]where $g$ and $h$ are induced by $e_{\sigma^k(i)}\mapsto\alpha_{\sigma^{k-1}(i)}\cdots\alpha_{\sigma(i)}\alpha_i$ and $e_i\mapsto\alpha_{\sigma^{n-1}(i)}\cdots\alpha_{\sigma^{k+1}(i)}\alpha_{\sigma^{k}(i)}$, respectively.

\textbf{Case 2.} $i\not\in\Delta$ and $\sigma^{k}(i)\in\Delta$.

By Lemma $\ref{basic2}(2)$, $\ell(P_{i}) = n + r$ for some $r\in\mathbb{N}^{+}$ such that $\sigma^{r}(i)\in\Delta$ and $k\geqslant r$.

If $k=r$, then $L(\sigma^{k}(i),n+r-k)=L(\sigma^{r}(i), n)$ which is projective by Lemma $\ref{basic2}(2)$. In this case, $\qpd_{A_{n,m}}(L(i,k))=\pd_{A_{n,m}}(L(i,k))=1$, due to Proposition \ref{qp=p1}.

Suppose $k>r$. Then $\sigma^{k}(i)\in\Delta$ and $\sigma^{n+r-k}(\sigma^k(i))=\sigma^{n+r}(i)=\sigma^r(i)\in\Delta$. By Case $1$, $L(\sigma^{k}(i),n+r-k)$ is periodic, and
so is $\Omega(L(i,k))$. Applying $(*)$ to $\Omega(L(i,k))$, we can construct a projective resolution of $L(i,k)$ as follows: \[\begin{tikzcd}
\cdots \arrow[r, "g_{1}"] & P_{\sigma^k(i)} \arrow[r, "h_{1}"] & P_{\sigma^{r}(i)} \arrow[r, "g_{1}"] & P_{\sigma^{k}(i)} \arrow[r, "f"] & P_{i} \arrow[r] & {L(i,k)} \arrow[r] & 0
\end{tikzcd}\]where $f$, $g_1$ and $h_1$ are induced by
$$e_{\sigma^{k}(i)}\mapsto \alpha_{\sigma^{k-1}(i)}\cdots\alpha_{\sigma(i)}\alpha_{i},\quad e_{\sigma^{r}(i)}\mapsto \alpha_{\sigma^{n+r-1}(i)}\cdots\alpha_{\sigma^{k+1}(i)}\alpha_{\sigma^{k}(i)},\quad
e_{\sigma^{k}(i)}\mapsto \alpha_{\sigma^{k-1}(i)}\cdots\alpha_{\sigma^{r+1}(i)}\alpha_{\sigma^{r}(i)},$$
respectively. Let $\beta: P_{\sigma^{r}(i)}\to P_i$ be the $A_{n,m}$-module homomorphism induced by
 $e_{\sigma^{r}(i)}\mapsto \alpha_{\sigma^{r-1}(i)}\cdots\alpha_{\sigma(i)}\alpha_{i}$. Then we obtain the following commutative diagram:
\[\begin{tikzcd}
\cdots \arrow[r, "h_{1}"] & P_{\sigma^{r}(i)} \arrow[r, "g_{1}"] \arrow[d, no head, shift right] \arrow[d, no head] & P_{\sigma^{k}(i)} \arrow[r, "h_{1}"] \arrow[d, no head, shift right] \arrow[d, no head] & P_{\sigma^{r}(i)} \arrow[r, "g_{1}"] \arrow[d, "\beta"] & P_{\sigma^{k}(i)} \arrow[r, "f"] \arrow[d] & P_{i} \arrow[r] \arrow[d] & 0 \arrow[d] \\
\cdots \arrow[r, "h_{1}"] & P_{\sigma^r(i)} \arrow[r, "g_{1}"]                                                    & P_{\sigma^{k}(i)} \arrow[r, "f"]                                                        & P_{i} \arrow[r]                                                         & 0 \arrow[r]                                & 0 \arrow[r]               & 0.
\end{tikzcd}\]
By Theorem \ref{main1.1}(4),  $\qpd_{A_{n,m}}(L(i,k))\leqslant1$.

\textbf{Case 3.} $i\in\Delta$ and $\sigma^{k}(i)\not\in\Delta$.

By Lemma $\ref{basic2}(2)$, $\ell(P_i)=n$ and $\ell(P_{\sigma^{k}(i)}) = n + r$ for some $r>0$ with $\sigma^r(\sigma^{k}(i))=\sigma^{r+k}(i)\in\Delta$. Then $\Omega(L(i,k))\simeq L\big(\sigma^{k}(i),\ell(P_{i})-k\big)=L(\sigma^{k}(i),n-k)$ which is non-projective by $n-k<n$. It follows that
\[
\Omega^2(L(i,k))\simeq \Omega(L(\sigma^{k}(i),n-k)) \simeq L(\sigma^n(i),r+k)=L(i, r+k).
\]
If $L(i, r+k)$ is projective, then $\qpd_{A_{n,m}}(L(i,k))=\pd_{A_{n,m}}(L(i,k))=2$. Now, assume that
$L(i, r+k)$ is non-projective. Since $\ell(P_i)=n$ and $L(i, r+k)$ is a quotient of $P_i$, we have $0<r+k<n$. Recall that $i\in\Delta$ and $\sigma^{r+k}(i)\in\Delta$. By Case $1$, the module $L(i, r+k)$ is periodic. Similarly, by $(*)$, we can construct a projective resolution of $L(i,k)$ as follows:
\[
\begin{tikzcd}
\cdots \arrow[r, "g_2"] & P_i \arrow[r, "h_2"] & P_{\sigma^{r+k}(i)} \arrow[r, "g_2"] & P_i \arrow[r, "f_1"] & P_{\sigma^k(i)} \arrow[r, "f_0"] & {P_{i}} \arrow[r] & L(i,k) \arrow[r] & 0,
\end{tikzcd}
\]where $f_0$, $f_1$, $g_2$ and $h_2$ are induced by
$$
e_{\sigma^k(i)}\mapsto\alpha_{\sigma^{k-1}(i)}\cdots\alpha_{\sigma(i)}\alpha_i,\quad
e_i\mapsto \alpha_{\sigma^{n-1}(i)}\cdots\alpha_{\sigma^{k+1}(i)}\alpha_{\sigma^{k}(i)},
$$
$$
e_{\sigma^{r+k}(i)}\mapsto\alpha_{\sigma^{r+k-1}(i)}\cdots\alpha_{\sigma(i)}\alpha_i,\quad e_i\mapsto\alpha_{\sigma^{n-1}(i)}\cdots\alpha_{\sigma^{r+k+1}(i)}\alpha_{\sigma^{r+k}(i)},$$
respectively. Further, let $\gamma: P_{\sigma^{r+k}(i)}\to P_{\sigma^k(i)}$ be the $A_{n,m}$-module homomorphism induced by
$e_{\sigma^{r+k}(i)}\mapsto\alpha_{\sigma^{r+k-1}(i)}\cdots\alpha_{\sigma^{k+1}(i)}\alpha_{\sigma^{k}(i)}$.
Then there is a  commutative diagram:
\[\begin{tikzcd}
\cdots \arrow[r, "g_2"] & P_i \arrow[r, "h_2"] \arrow[d, no head, shift right] \arrow[d, no head] & P_{\sigma^{r+k}(i)} \arrow[r, "g_2"] \arrow[d, "\gamma"] & P_i \arrow[r, "f_1"] \arrow[d, no head, shift right] \arrow[d, no head] & P_{\sigma^k(i)} \arrow[r, "f_0"] \arrow[d] & P_i \arrow[r] \arrow[d] & 0 \arrow[d] \\
\cdots \arrow[r, "g_2"] & P_i \arrow[r, "f_1"]                                                  & P_{\sigma^k(i)} \arrow[r, "f_0"]                     & P_i \arrow[r]                                                           & 0 \arrow[r]                                & 0 \arrow[r]             & 0.
\end{tikzcd}\]
By Theorem \ref{main1.1}(4), $\qpd_{A_{n,m}}(L(i,k))\leqslant2$.

\textbf{Case 4.} $i\not\in\Delta$ and $\sigma^{k}(i)\not\in\Delta$.

By Lemma $\ref{basic2}(2)$, $\ell(P_{i}) = n + r_1$ and $\ell(P_{\sigma^{k}(i)}) = n + r_2$ for $0<r_1<n$ and  $0<r_2<n$ such that $\sigma^{r_1}(i), \sigma^{r_2+k}(i)\in\Delta$. By Lemma $\ref{basic2}(4)$,  $\Omega(L(i,k))\simeq L(\sigma^{k}(i),n+r_1-k)$.

If either $\Omega(L(i,k))$ or  $\Omega^2(L(i,k))$ is projective, then $\pd_{A_{n,m}}(L(i,k))\leqslant2$.
In this case, $\qpd_{A_{n,m}}(L(i,k))=\pd_{A_{n,m}}(L(i,k))\leqslant 2$ by Proposition \ref{qp=p1}.

Now, assume that both $\Omega(L(i,k))$ and $\Omega^2(L(i,k))$ are non-projective.
Then $n+r_1-k<\ell(P_{\sigma^{k}(i)})=n+r_2$. It follows that
$$\Omega^2(L(i,k))\simeq \Omega\big(L(\sigma^{k}(i),n+r_1-k)\big)\simeq L(\sigma^{n+r_{1}}(i),r_{2}-r_{1}+k)=L(\sigma^{r_1}(i),r_{2}-r_1+k).$$
By $\sigma^{r_1}(i)\in\Delta$, we have $\ell(P_{\sigma^{r_1}(i)})=n$ by Lemma $\ref{basic2}(2)$.
Since $\Omega^2(L(i,k))$ is non-projective, it is clear that $0<r_2-r_1+k<n$.
Note that $\sigma^{r_1}(i)\in\Delta$ and $\sigma^{r_2-r_1+k}(\sigma^{r_1}(i))=\sigma^{r_2+k}(i)\in\Delta$. Thus $L(\sigma^{r_1}(i),r_{2}-r_{1}+k)$ is periodic by Case $1$. Still by $(*)$, we can construct a projective resolution of $L(i,k)$ as follows:
\[
\begin{tikzcd}
\cdots \arrow[r, "g_3"] & P_{\sigma^{r_{1}}(i)} \arrow[r, "h_3"] & P_{\sigma^{r_2+k}(i)} \arrow[r, "g_3"] & P_{\sigma^{r_{1}}(i)} \arrow[r, "d_1"] & P_{\sigma^{k}(i)} \arrow[r, "d_0"] & P_{i} \arrow[r] & L(i,k) \arrow[r] & 0
\end{tikzcd}\]
where $d_0, d_1, g_3$ and $h_3$ are induced by
$$
e_{\sigma^k(i)}\mapsto\alpha_{\sigma^{k-1}(i)}\cdots\alpha_{\sigma(i)}\alpha_{i},\quad e_{\sigma^{r_{1}}(i)}\mapsto\alpha_{\sigma^{n+r_{1}-1}(i)}\cdots\alpha_{\sigma^{k+1}(i)}\alpha_{\sigma^{k}(i)},
$$
$$
e_{\sigma^{r_2+k}(i)}\mapsto\alpha_{\sigma^{r_2+k-1}(i)}\cdots\alpha_{\sigma^{r_1+1}(i)}\alpha_{\sigma^{r_1}(i)},\quad e_{\sigma^{r_{1}}(i)}\mapsto\alpha_{\sigma^{n+r_{1}-1}(i)}\cdots\alpha_{\sigma^{r_2+k+1}(i)}\alpha_{\sigma^{r_2+k}(i)},
$$
Further, let $\delta: P_{\sigma^{r_2+k}(i)}\to P_{\sigma^{k}(i)}$ and $\epsilon: P_{\sigma^{r_{1}}(i)}\to P_i$ be the $A_{n,m}$-module homomorphisms induced by
$$e_{\sigma^{r_{2}+k}(i)}\mapsto\alpha_{\sigma^{r_2+k-1}(i)}\cdots\alpha_{\sigma^{k+1}(i)}\alpha_{\sigma^{k}(i)}
\quad e_{\sigma^{r_1}(i)}\mapsto\alpha_{\sigma^{r_1-1}(i)}\cdots\alpha_{\sigma(i)}\alpha_{i},$$
respectively. Then there is a commutative diagram:
\[
\begin{tikzcd}
\cdots \arrow[r, "g_3"] & P_{\sigma^{r_{1}}(i)} \arrow[r, "h_3"] \arrow[d, no head, shift right] \arrow[d, no head] & P_{\sigma^{r_2+k}(i)} \arrow[r, "g_3"] \arrow[d, "\delta"] & P_{\sigma^{r_{1}}(i)} \arrow[r, "d_1"] \arrow[d, "\epsilon"] & P_{\sigma^{k}(i)} \arrow[r, "d_0"] \arrow[d] & P_{i} \arrow[r] \arrow[d] & 0 \arrow[d] \\
\cdots \arrow[r, "g_3"] & P_{\sigma^{r_{1}}(i)} \arrow[r, "d_1"]                                                  & P_{\sigma^{k}(i)} \arrow[r, "d_0"]             & P_{i} \arrow[r]             & 0 \arrow[r]                                      & 0 \arrow[r]                              & 0.
\end{tikzcd}
\]
Thus $\qpd_{A_{n,m}}(L(i,k))\leqslant2$ by Theorem \ref{main1.1}(4).
\end{proof}

\textbf{Proof of Theorem \ref{main}.}
$(1)$ Without loss of generality, we can assume $\lambda_1=1$. It can be checked that
\[
\pd_{A_{n,m}}(S_{i}) =
\begin{cases}
2, & \text{if } i = 1, \\
1, & \text{otherwise}.
\end{cases}
\]

$(2)$ $A_{n,n}$ is a self-injective algebra of finite representation type. By Remark \ref{TSC}, $\qgd(A_{n,n})=0$.

$(3)$ Note that $\{L(i,k)\mid 1\leqslant i\leqslant n, 1\leqslant k\leqslant \ell(P_i)\}$ is the set of
isomorphism classes of indecomposable $A_{n,m}$-modules. By Theorem \ref{QPP}(1), $$\qgd(A_{n,m})=\sup\{\qpd_{A_{n,m}}(L(i,k))\mid 1\leqslant i\leqslant n, 0< k\leqslant \ell(P_i)\}.$$
Combining Lemma \ref{pdfinite} with Lemma \ref{Calculation}, we have $\qpd_{A_{n,m}}(L(i,k))\leqslant 2$ for $1\leqslant i\leqslant n$ and $0<k\leqslant \ell(P_i)$. This implies $\qgd(A_{n,m})\leqslant 2$.
Moreover, since $\Delta\not=\{1,2,\cdots,n\}$, there exists an indecomposable projective $A_{n,m}$-module $P_i$ such that $\ell(P_i)>n$ by Lemma \ref{basic2}(2). It follows from Lemma \ref{pdfinite} that $\qpd_{A_{n,m}}(L(i,n))=\pd_{A_{n,m}}(L(i,n))=2$. Thus $\qgd(A_{n,m})=2$. Since $n>2$ and $1<m<n$, there exists an integer $i\in\Delta$ such that $\sigma^{k}(i)\in\Delta$ for some integer $k$ with $0<k<n$.
According to Case $1$ in the proof of Lemma \ref{Calculation}, the module $L(i,k)$ is periodic but non-projective. This implies $\gd(A_{n,m})=\infty$.
\hfill $\square$

\medskip
{\bf Acknowledgements.} The research work was partially supported by the National Natural Science Foundation of China (Grant 12031014).

{\footnotesize
\medskip

Hongxing Chen,

School of Mathematical Sciences  \&  Academy for Multidisciplinary Studies, Capital Normal University, 100048 Beijing,

P. R. China

{\tt Email: chenhx@cnu.edu.cn (H.X.Chen)}

\smallskip

Xiaohu Chen,

School of Mathematical Sciences, Capital Normal University, 100048
Beijing, P. R. China

{\tt Email: xiaohu.chen@cnu.edu.cn (X.H.Chen)}

\smallskip

Mengge Liu,

School of Mathematical Sciences, Capital Normal University, 100048
Beijing, P. R. China;

{\tt Email: 1602567185@qq.com (M.G.Liu)}}


\begin{thebibliography}{99}

\bibitem{2006ASS}{{\sc  I. Assem}, {\sc  D. Simson}, {\sc  A. Skowroński}, \emph{Elements of the
representation theory of associative algebras. Vol. 1. Techniques of representation theory}, London Mathematical Society Student Texts, {\bf 65}, Cambridge University Press, Cambridge, 2006.}

\bibitem{1961auslander} {{\sc  M. Auslander}, Modules over unramified regular local rings, \textit{Illinois
J. Math.} \textbf{5} (1961) 631–647.}

\bibitem{2021}{{\sc B. Briggs}, {\sc E. Grifo} and {\sc J. Pollitz}, Constructing nonproxy small test
modules for the complete intersection property, {\it Nagoya Math. J.} {\bf 246} (2022) 412-429.}

\bibitem{xcf}{{\sc H. X. Chen}, {\sc M. Fang} and {\sc C. C. Xi}, Tachikawa's second conjecture, derived
recollements, and gendo-symmetric algebras, \emph{Compos. Math.} \textbf{160} (2024) 2704-2734.}

\bibitem{2012Chen}{{\sc H. X. Chen}, {\sc S. Y. Pan} and {\sc C.C. Xi}, Inductions and restrictions for
stable equivalences of Morita type, {\it J. Pure Appl. Algebra} {\bf 216} (2012) 643-661.}

\bibitem{xc}{{\sc H. X. Chen} and {\sc C. C. Xi}, Homological theory of self-orthogonal modules,
\textit{Trans. Amer. Math. Soc.} {\bf 378} (2025) 7287-7335.}

\bibitem{G2024}{{\sc M. Gheibi}, Quasi-injective dimension, {\it J. Pure Appl. Algebra} {\bf 228} (2024),
Paper No. 107468.}

\bibitem{QPD}{{\sc M. Gheibi}, {\sc D. A. Jorgensen} and {\sc R. Takahashi}, Quasi-projective dimension,
{\it Pacific J. Math.} {\bf 312} (2021) 113-147.}

\bibitem{2024}{{\sc V. H. Jorge-p\'{e}rez}, {\sc P. Martins} and {\sc V. D. Mendoza-rubio},
A generalized depth formula for modules of finite quasi-projective dimension, arXiv:2409.08996v2, 2024.}

\bibitem{1994liu} {{\sc S. P. Liu} and {\sc R. Schulz}, The exsitence of bounded infinite  \textit{D}Tr-orbits. \textit{Proc. Amer. Math. Soc.} \textbf{122} (1994) 1003-1005.}

\bibitem{1991Rickard} {{\sc J. Rickard}, Derived equivalences as derived functors, \textit{J. London Math. Soc. (2)} \textbf{43} (1991) 37-48.}

\bibitem{NMT}{{\sc N. M. Tri}, Chouinard's formula for quasi-injective dimension,
\emph{Rev. R. Acad. Cienc. Exactas Fís. Nat. Ser. A Mat. RACSAM } \textbf{119} (2025) Paper No. 52.}

\bibitem{1994Weibel}{{\sc C. A. Weibel}, \emph{An introduction to homological algebra},  Cambridge Studies in Advanced Mathematics \textbf{38}, Cambridge University Press, Cambridge, 1994.}

\bibitem{xi18}{{\sc C. C. Xi}, Derived equivalences of algebras, \emph{Bull. Lond. Math. Soc.} \textbf{50} (2018) 945-985.}

\end{thebibliography}
\end{document}